\documentclass[10pt,oneside,openany,notitlepage]{article}
\usepackage[english]{babel}

\usepackage{graphics}
\usepackage[dvips]{color}
\usepackage[pdftex]{graphicx}
\usepackage{multicol}
\usepackage{times}
\usepackage{amssymb}
\usepackage{amsmath}
\usepackage{flafter}
\usepackage[center]{caption2}
\usepackage{indentfirst,latexsym,bm}
\usepackage{picinpar}
\usepackage[numbers,sort&compress]{natbib}
\usepackage[toc,page,title,titletoc,header]{appendix}
\usepackage{dsfont}
\usepackage{mathrsfs}
\usepackage{lipsum}
\usepackage{nonfloat}
\usepackage{sectsty}

\sectionfont{\fontsize{12}{11}\selectfont}

\newtheorem{remark}{Remark}[section]
\newtheorem{theorem}{Theorem}[section]

\newtheorem{lemma}{Lemma}[section]
\newtheorem{proposition}{Proposition}[section]
\newtheorem{corollary}{Corollary}[section]

\newtheorem{definition}{Definition}[section]

\newenvironment{proof}
               {\begin{sloppypar} \noindent{\bf Proof}}
               {\hspace*{\fill} $\square$ \end{sloppypar}}

\newcommand{\beq}[1]{\begin{equation} \label{#1}}
\newcommand{\eeq}{\end{equation}}

\newcommand\myfigure[1]{%
\medskip\noindent\begin{minipage}{\columnwidth}
\centering%
#1%
\end{minipage}\medskip}

\makeatletter
\def\@biblabel#1{#1.\hfill}
\makeatother

\title{\bfseries\textrm{Critical Connectivity and Fastest Convergence Rates of Distributed Consensus with Switching Topologies and Additive Noises*} \footnotetext{*The research of G. Chen and C. Chen was supported by the National Key Basic Research Program of China (973 program) under grant 2014CB845301/2/3, and by the National Natural Science Foundation of China under grants 91427304 and 61673373.
The research of L.Y. Wang and G. Yin was supported in part by U.S.
Army Research Office under grant W911NF-15-1-0218.}}
\date{}
\author{Ge Chen,\thanks{National Center for Mathematics and Interdisciplinary Sciences \& Key Laboratory of Systems and
Control, Academy of Mathematics and Systems Science, Chinese Academy of Sciences, Beijing 100190,
China, {\tt
chenge@amss.ac.cn}} \and Le Yi Wang,\thanks{Department of Electrical
and Computer Engineering,
Wayne State University, Detroit, MI 48202, USA, {\tt
lywang@ece.eng.wayne.edu}}
\and Chen Chen,\thanks{is with the Shannon Lab, 2012 Laboratory, Huawei Technologies Co. Ltd., {\tt
chenchen9@huawei.com}} \and
George Yin\thanks{Department of
Mathematics, Wayne State University, Detroit, MI 48202, USA, {\tt
gyin@math.wayne.edu}}
}

\begin{document}
\begin{multicols}{2}
[
\maketitle
]

\textbf{Abstract-} Consensus conditions and convergence speeds are crucial for
distributed consensus algorithms of networked systems.
 Based on a basic first-order average-consensus protocol with time-varying topologies and additive noises,
 this paper first investigates its critical consensus condition on network topology by stochastic approximation frameworks. A new  joint-connectivity condition called extensible joint-connectivity that contains
 a parameter $\delta$ (termed the extensible exponent) is proposed. With this and a balanced topology condition,  we show that a critical value
 of $\delta$ for consensus is $1/2$.
Optimization on convergence rate of this protocol is further investigated.
It is proved that the fastest convergence rate, which is the theoretic optimal
rate among all controls,
 is of the order $1/t$ for the best topologies, and
 is of the order $1/t^{1-2\delta}$ for the worst topologies which are balanced and satisfy the extensible joint-connectivity condition.
For practical implementation, certain open-loop control strategies are introduced
to achieve consensus with a convergence rate of
the same order as the fastest convergence rate.
Furthermore, a consensus condition is derived for
non stationary and strongly correlated random topologies.
The algorithms and consensus conditions are applied to distributed consensus computation of mobile ad-hoc networks; and their related
 critical exponents are derived from relative velocities of mobile agents for guaranteeing consensus.


\section{Introduction}\label{intro}

Consensus of multi-agent systems has drawn
considerable attention from
various fields
over the past two  decades.
For example, physicists investigate the synchronization phenomena of coupled oscillators, flashing fireflies, and chirping crickets \cite{Kuramoto1975,Acebron2004};
  biologists, physicists, and computer scientists try to understand and model the flocking phenomenon of animals' behavior \cite{Buhl2006,Vicsek1995,Reynolds1987}; sociologists simulate the emergence and spread of public opinions \cite{Deffuant2000,Hegselmann2002}.
  Because of the importance, effort has been devoted to
the mathematical analysis of consensus of flocks
  \cite{Jad2003,Smale2007,Liu2009,Chen2014}.
 Meanwhile
 consensus control algorithms have been developed for a wide range of applications, such as formation control of robots and
 vehicles \cite{Savkin2004,Fax2004,Ren2008}, attitude synchronization of rigid bodies and multiple spacecrafts \cite{Bai2008,Abdessameud2009,Sarlette2009}, and distributed computation, filtering and resource allocations of  networked systems \cite{Lynch1996,Saber2005,Yin2011}.
 The common thread in the consensus research is a group of agents with interconnecting
neighbor graphs trying to achieve a global coordination or collective behavior by using neighborhood information permitted
by the network topologies.

Although there are many
interesting consensus protocols like second-order and fractional-order consensus protocols \cite{Lu2016},
the first-order average-consensus protocol is the most basic one. This protocol usually assumes that the network topologies cannot be directly controlled and each node only knows its own and neighbors' information, and has been investigated
using different approaches
to accommodate
different kinds of uncertainties. For example, in wireless communication networks, channel reliability
 is affected by thermal noise, channel fading, and signal quantization; in formation of multiple satellites, vehicles or robots, there exist measurement noises
in observations of neighbors' states. To model random failures of communication links, some papers use deterministic network
topologies but allow their switching
\cite{Saber2004,Xiao2004,Olshevsky2011}; and others adopt stochastic settings in which network topologies evolve according to some random distributions
 \cite{Hatano2005,Porfiri2007,Akar2008,Zhou2009,Tahbaz2010,Touri2011}. We remark that most models in the existing research effort do not contain observation noise, and as such they do not cover scenarios of
 measurement noise  or quantization error. To overcome this deficiency,
 some papers consider the first-order average-consensus protocols with additive noises \cite{Touri2009,Patterson2010,Aysal2010,Young2010,Stankovic2011,Huang2009,Huang2010,Huang2012,Li2009,Li2010,Yin2011,Yin2013,Wang2013},
 among which a common feature
 is in using stochastic approximation methodologies.

 The main idea of
 distributed stochastic approximation is: Each agent in the network uses a decreasing gain function acting on the information received from its neighbors to reduce the impact of communication or measurement noises. Using this idea,
Huang and Manton  \cite{Huang2009} and Li and Zhang \cite{Li2009} considered first-order discrete-time and continuum-time consensus models with fixed topology and additive noise, respectively, and showed that the algorithms could achieve consensus in a probability sense if the topologies were balanced and connected.
Later, this consensus condition was
relaxed from fixed balanced topologies to switched balanced topologies that  satisfied a \emph{uniform joint-connectivity} condition  \cite{Touri2009,Li2010}, namely, the union of the topology graphs over a given bounded time interval was connected uniformly in time, and further relaxed to general directed topologies, which may not be balanced,  with uniform joint-connectivity  \cite{Huang2012}. Huang  \cite{Huang2010} also applied stochastic approximation methods to consensus problems for networks  over lossy wireless networks containing random link gains, additive noises and Markovian lossy signal receptions.
On the other hand, motivated by resource allocation problems in computing, communications, inventory, space, and power generations,
Yin, Sun and Wang  \cite{Yin2011} introduced a stochastic approximation algorithm for constrained consensus problems of networked systems, where
consensus conditions were established by assuming that the topologies were randomly switched under a Markov chain framework. This algorithm was further
investigated and expanded later \cite{Yin2013,Wang2013}.

Despite the existing research work on first-order average-consensus protocols, some key problems remain unsolved.
\cite{Touri2011} showed
that first-order average-consensus protocols with deterministic topologies and no additive noises could achieve consensus if and only if the time-varying topologies satisfied an \emph{infinite joint-connectivity} condition, i.e., the union of the topologies from any finite time to infinite was connected, providing all topologies had the same stationary distribution. However, under the same protocol but with additive noises the current best condition on topologies for consensus is the uniform joint-connectivity \cite{Touri2009,Li2010,Huang2012}. Thus, there exists a
critical gap between the consensus conditions on topologies with and without additive noises. Naturally, for protocols with additive noises, an
open question
is: What is the critical connectivity condition on topologies for consensus?

To address this problem  we propose a new condition for topologies named \emph{extensible joint-connectivity} in this paper.
This condition allows the length of the interval during which the union of the network topologies is connected to increase as the time grows with a
 growing rate,  called \emph{extensible exponent} and denoted by $\delta$. This is an intermediate condition between the uniform joint-connectivity and infinite joint-connectivity.
Furthermore, we use a
stochastic approximation approach to attenuate  noise effect and treat the gain function as the
control input.
Under the extensible joint-connectivity and balanced topology condition it will be shown that if $\delta\leq 1/2$, for all topology sequences there exist open-loop controls of the gain function such that the system reaches consensus; if $\delta>1/2$, there exist some topology sequences such that no open-loop control of the gain function can make the system to reach consensus.

One of the most basic and important tasks on multi-agent systems  is to optimize their performance. However,
the current theoretical research on this topic is still at an early
stage \cite{Imer2006,Yuksel2007}.
For example, the convergence speed is an important performance of average-consensus protocols which has been considered under both fixed and switching topologies \cite{Olshevsky2011,Touri2011,Patterson2010,Yin2011}. However  its
 optimization currently focuses on the case of fixed topology by iterate averaging of stochastic approximation \cite{Yin2013} or optimization algorithms of static graphs \cite{Xiao2004,Kim2006}.
These methods cannot
be used to optimize average-consensus protocols under uncontrollable time-varying network topologies, leaving an open problem on how to optimize the convergence rate of average-consensus protocols \cite{Cao2012}.

This paper investigates this problem for the first time, still based on the stochastic approximation methodology. Since the convergence rate depends on the uncontrollable time-varying topologies whose global information is unavailable,
its optimization is difficult  and the traditional optimization theory for stochastic approximation cannot cover this scenario.
This paper is concentrated on
the \emph{fastest convergence rate}, which is the theoretical optimal convergence rate  among all gain functions, with respect to the best and worst topology sequences, respectively.  It will be shown that the fastest convergence rate is of the order $1/t$, in the sense of
$\Theta(1/t)$) for the best topologies, and $\Theta(1/t^{1-2\delta})$ for the worst topologies which are balanced and satisfy the extensible joint-connectivity condition. These results indicate that for any balanced and uniformly jointly connected topologies the fastest convergence rate is $\Theta(1/t)$.
For implementation on practical systems, this paper presents
some open-loop controls for this protocol to reach consensus with a convergence rate of
the same order as the fastest convergence rate.

Finally,
in many practical systems, especially those involving wireless communications,
network topologies of distributed consensus protocols are random networks.
This kind of algorithms has been investigated in many papers. However almost all of them assume that the topologies are either an i.i.d. sequence or a stationary Markov process \cite{Hatano2005,Porfiri2007,Akar2008,Zhou2009,Tahbaz2010,Huang2010,Yin2011}. These assumptions may not fit some practical situations such as
 mobile wireless sensor networks or multi robot systems whose topologies depend on the distances among agents and consequently may be non stationary.
Thus, the last question studied in this paper is: for the average-consensus protocol with random topology and additive noise, can we relax the topology condition for consensus to non stationary and strongly correlated sequences?

This paper gives an answer to this question
by proposing a consensus condition in which
 the random topology sequence can be strongly correlated and its connectivity probability can be a negative power function.
 To illustrate relevance of this condition,
this result is applied to distributed consensus computation of a mobile ad-hoc network. In this application it is shown that
 to guarantee consensus the distance between mobile agents cannot grow too fast. To be specific, if the velocity difference between agents is of the order $\frac{1}{t}$, with average-consensus protocols the mobile ad-hoc network can reach a consensus state; on the other hand, from simulations it is 
 demonstrated that if the velocity difference is bigger than $\frac{c}{t^b}$, where  $c>0$ and $b\in[0,1)$ are two constants, then the mobile ad-hoc network cannot reach a consensus state.

The rest of the paper is organized as follows: Section \ref{prepa} introduces our consensus protocol and some basic definitions. Section \ref{Consensus_condition} introduces the critical connectivity condition of network topologies for consensus.  Section
\ref{Convergence_speed} investigates the fastest convergence rates with respect to the best and worst topologies, respectively.
In Section \ref{sec_stoc} we present a consensus condition for random network topologies and an application to mobile ad-hoc networks.
 Finally, we conclude this paper with discussions on the main findings of this paper.

\section{Preliminaries}\label{prepa}
\subsection{Definitions in Graph Theory}
Let $\mathcal{G}=\{ \mathcal{V }, \mathcal{E},\mathcal{A}\}$ be a weighted digraph, where $\mathcal{V}=\{1,2,\ldots n\}$ is the set of nodes with node
$i$ representing the $i$th agent, $\mathcal{E}$ is the set of edges, and $\mathcal{A}\in \mathbb{R}^{n\times n}$ is the
weight matrix. An edge in $\mathcal{G}$ is
an
ordered pair $(j,i)$, and $(j,i)\in \mathcal{E}$ if and only if the $i$th agent can
receive information from the $j$th agent directly. The neighborhood of
the $i$th agent is denoted by $\mathcal{N}_i=\{j\in \mathcal{V }| (j,i)\in \mathcal{E}\}$. An
element of $\mathcal{N}_i$ is called a neighbor of $i$. $\mathcal{G}$ is called  an undirected graph when all of its edges are bidirectional,
which means $j\in\mathcal{N}_i$ if and only if $i\in\mathcal{N}_j$.
Let $a_{ij}>0$ denote the weight of the edge $(j,i)$. The weight matrix  $\mathcal{A}$ is defined by $\mathcal{A}_{ij}=a_{ij}$
for $(j,i)\in\mathcal{E}$, and $\mathcal{A}_{ij}=0$ otherwise.

 For graph $\mathcal{G}$,
the in-degree of $i$ is defined as $\deg_{\rm{in}}^i=\sum_{j\in \mathcal{N}_i} a_{ij}$ and
the out-degree of $i$ is defined as $\deg_{\rm{out}}^i=\sum_{i\in\mathcal{N}_j}a_{ji}$.
If $\deg_{\rm{in}}^i=\deg_{\rm{out}}^i$ for all $1\leq i\leq n$, we call $\mathcal{G}$  a \emph{balanced graph}.
 The
Laplacian matrix of $\mathcal{G}$ is defined as $L_{\mathcal{G}}=D_{\mathcal{G}}- \mathcal{A}_{\mathcal{G}}$,
where
$D_{\mathcal{G}}=\mbox{diag}(\deg_{in}^1, \ldots,\deg_{in}^n)$, and $[\mathcal{A}_{\mathcal{G}}]_{ij}$ equals to $a_{ij}$ if $j\in \mathcal{N}_i$ and $0$ otherwise.

A sequence $(i_1, i_2), (i_2, i_3), \ldots, (i_{k-1}, i_k)$  of edges is called
a directed path from node $i_1$ to node $i_k$.
 $\mathcal{G}$ is called a strongly
connected digraph, if for any $i,j \in \mathcal{V}$, there is a directed path
from $i$ to $j$. A strongly connected undirected graph is also called
a connected graph. For graphs $\mathcal{G}(t)=\{V,\mathcal{E}(t),\mathcal{A}(t)\}$, $i\leq t <j$, their union is defined by $\cup_{i\leq t<j}\mathcal{G}(t):=\{V,\cup_{i\leq t<j}(\mathcal{E}(t),\mathcal{A}(t))\}$.
Note that there may exist multiple weighted edges from one vertex to another in $\cup_{i\leq t<j}\mathcal{G}(t)$.

\subsection {Consensus Protocol}
This paper considers a discrete-time first-order system containing $n$ agents, where 
agent $i$'s state $x_i(t)$ is updated by
 \begin{eqnarray}\label{model0}
x_i(t+1)=x_i(t)+u_i(t),~~t=1,2,\ldots
\end{eqnarray}
 Here  $u_i(t)$ is the control input of the $i$th agent. For simplicity, we suppose $x_i(t)$ and $u_i(t)$ are scalars. As mentioned above, this paper will investigate a basic stochastic approximation algorithm for average-consensus, that is, the control $u_i(t)$ is chosen by
\begin{equation}\label{model1}
u_i(t)=a(t)\sum_{j\in\mathcal{N}_i(t)} a_{ij}^t \left[x_j(t)-x_i(t)+w_{ji}(t)\right],
\end{equation}
where $a(t)\geq 0$ is the common gain control at time $t$, $\mathcal{N}_i(t)$ is the neighbors of node $i$ at time $t$, $a_{ij}^t$ is the weight of the edge
$(j,i)$ at time $t$, and $w_{ji}(t)$ is the noise of agent $i$ receiving information from agent $j$ at time $t$. Throughout this paper, we assume
\begin{eqnarray}\label{a_assum}
1\leq a_{ij}^t \leq a_{\max}, ~~t\geq 1, 1\leq i\leq n, j\in\mathcal{N}_i(t).
\end{eqnarray}
This kind of protocols was investigated in several papers \cite{Touri2009,Huang2009,Huang2012,Li2009,Li2010}
with
applications to distributed computation of wireless sensor networks and formation control of multiple satellites, vehicles, or robots.

We define the $\sigma$-algebra generated by the noises
$w_{ji}(k),1\leq k\leq t,1\leq i\leq n, j\in\mathcal{N}_i(k)$ by $$\mathcal{F}_t=\sigma(w_{ji}(k),1\leq k\leq t,1\leq i\leq n, j\in\mathcal{N}_i(k)).$$ The probability space of system (\ref{model0})-(\ref{model1}) is
$(\Omega,\mathcal{F}_{\infty},P)$.
$\mathcal{G}(t)=(\mathcal{V},\mathcal{E}(t),\mathcal{A}(t))$ represents the topology of  system  (\ref{model0})-(\ref{model1}) at time $t$, where
$\mathcal{E}(t)=\{(j,i)|j\in\mathcal{N}_i(t)\}$ is the edge set of $\mathcal{G}(t)$, and $\mathcal{A}(t)$ is the weight matrix satisfying
 $\mathcal{A}_{ij}(t)=a_{ij}^t$ if $(j,i)\in\mathcal{E}(t)$ and $\mathcal{A}_{ij}(t)=0$ otherwise. The corresponding topology sequence is $\{\mathcal{G}(t)\}_{t\geq 1}=\{(\mathcal{V},\mathcal{E}(t),\mathcal{A}(t))\}_{t\geq 1}$.
For simplicity, we use $L(t)=L_{\mathcal{G}(t)}$ for the Laplacian matrix of $\mathcal{G}(t)$.
Let $x(t)=(x_1(t),x_2(t),\ldots,x_n(t))'$ where $z'$ denotes the transpose of $z$.
The protocol (\ref{model0})-(\ref{model1}) can be rewritten as the following matrix form:
\begin{eqnarray}\label{model3}
x(t+1)=[I-a(t)L(t)]x(t)+a(t)\widehat{w}(t),~t\geq 1,
\end{eqnarray}
where $\widehat{w}(t)\in\mathbb{R}^n$ whose $i$-th element is $\sum_{j\in\mathcal{N}_i(t)} a_{ij}^t w_{ji}(t)$.

In this paper, we assume that there
is no
 central controller who knows the global information of the evolution of  the system, and that
the so-called consensus control
requires to design off-line gains $a(t)$
 such that all agents achieve an agreement on their states in mean square sense, when $t\rightarrow \infty $. The mean square consensus is defined as follows.

 \begin{definition}\label{consensus_def}\cite{Li2010}
We say the system (\ref{model0})-(\ref{model1})
reaches \emph{mean square consensus} if $($i$)$ there exists
 a random variable $x^*$ satisfying $|E(x^*)|<\infty$ and Var$(x^*)<\infty$ such that
$\lim_{t\rightarrow \infty}E\|x(t)-x^*\mathds{1}\|^2=0,$
  where  $\mathds{1}\in \mathbb{R}^n$ is the column vector of all 1s,
  and
  $($ii$)$ reaches \emph{unbiased mean square average-consensus} if
  in addition, $x^*$ satisfies $E(x^*)=\frac{1}{n}\sum_{i=1}^n x_i(1)$.
 \end{definition}

 \subsection{Standard Notation}
 The following standard mathematical notation will be used in this paper.  Given a random variable $X$, let $E[X]$ and Var$(X)$
 be its expectation and variance, respectively.
For a vector $Y$, let $Y_i$ denote its $i$th entry.  For a real number $x$,
$\lfloor x \rfloor$ is the maximum integer less than or equal to $x$, and $\lceil x \rceil$ is the smallest integer larger than or equal to $x$. Let $\|\cdot\|$ denotes the $l_2-$norm (Euclidean norm).
Given two sequences of
positive numbers $g_1(t)$, $g_2(t)$,
 \begin{itemize}
\item $g_1(t)=O(g_2(t))$ if there exists a constant $c>0$ and a
value $t_0>0$ such that $g_1(t)\leq c g_2(t)$ for all $t\geq t_0$.
\item $g_1(t)=\Theta(g_2(t))$ if there exist two constants $c_2>c_1>0$ and a
value $t_0>0$ such that $c_1 g_2(t) \leq g_1(t)\leq c_2 g_2(t)$ for all $t\geq t_0$.
\item $g_1(t)=o(g_2(t))$ if
$\lim_{t\rightarrow\infty} g_1(t)/g_2(t)=0$.
\end{itemize}

\section{Critical Connectivity for Consensus}\label{Consensus_condition}
This section  provides some consensus conditions for  system (\ref{model0})-(\ref{model1}). In Subsection \ref{EJC}
we will propose a new condition concerning the connectivity, while the sufficient and necessary conditions of consensus are given
in Subsections \ref{suff_sec} and \ref{nece_sec}, respectively.

\subsection{Extensible Joint-connectivity}\label{EJC}


The uniform joint-connectivity  of the topologies is a widely used condition in the consensus research of multi-agent systems. However, this condition is not robust for some situations. For example, if the links in a networked system have a positive probability of failure, it can be computed that with probability $1$ the uniform joint-connectivity condition is not satisfied. Also,
this condition cannot be satisfied in some flocking models \cite{Chen2015}. To accommodate practical uncertainties in networked systems,
we propose a new condition for topologies,  called \emph{extensible joint-connectivity}, as follows:

{\bf (A1)} There exist constants $\delta \geq 0$ and $c\geq 1
$, and an infinite
sequence $1=t_1<t_2<t_3<\cdots$ such that $t_k\leq t_{k-1}+c t_{k-1}^{\delta}$ and $\cup_{t_{k-1}\leq t<t_k}\mathcal{G}(t)$ is strongly connected for all $k>1$.

\vskip 1mm

In (A1), we call $\delta$
the \emph{extensible exponent} for the joint-connectivity of the topologies.
For any $\delta$, (A1) is stronger than the infinite joint-connectivity assumption, which can be formulated by $\bigcup_{t\geq k} \mathcal{G}(k)$ being strongly connected for all $k\geq 1$. On the other hand, for any positive $\delta$, (A1) is weaker than the uniform joint-connectivity assumption. In fact, the uniform
joint-connectivity is a special case of (A1) with $\delta=0$.

\begin{remark}\label{rem_adv}
Compared to the uniform joint-connectivity, one advantage of the extensible joint-connectivity is that it can be used to analyze systems with random topology, even if the  probability of connectivity  of the topology
is not stationary and
decays in a negative power rate. In this case (by (\ref{rand_topo_3_1}) in Appendix \ref{App_rand_topo}) with probability $1$ there exists a finite time $T>0$ such that (A1) is satisfied for all $t\geq T$.
 This property has been applied to distributed consensus computation of mobile ad-hoc networks in Subsection \ref{subsec_app}, where
 the probability of successful communications between two agents depends on their distance.
\end{remark}

\subsection{Sufficient Conditions for Consensus}\label{suff_sec}

We first give a key lemma deduced from \cite{Touri2011}.
Before the statement of this key lemma some definitions are needed.
For protocol (\ref{model3}), define
$$\Phi(t,i):=\left[I-a(t) L(t)\right]\cdots \left[I-a(i) L(i)\right].$$
Take $\prod_{l=i}^{t}(\cdot):=I$ when $t<i$.
For any $x\in\mathbb{R}^n$,
let $x_{\rm{ave}}:=\frac{1}{n}\sum_{i=1}^n x_i$ be the average value of $x$,
and define
\begin{eqnarray*}
V(x):=\|x-x_{\rm{ave}}\mathds{1}\|^2=\sum_{i=1}^n (x_i-x_{\rm{ave}})^2.
\end{eqnarray*}
 For an integer sequence $\{t_k\}_{k\geq 1}$,
denote
\begin{eqnarray*}
k^i:=\min\{k: t_{k}\geq i+1\} \mbox{~and~} \widetilde{k}^t:=\max\{k: t_k-1\leq t\}.
\end{eqnarray*}
 Set $$d_{\max}:=\sup_{t,i}\sum_{j\in\mathcal{N}_i(t)}a_{ij}^t \leq (n-1)a_{\max}.$$
Also, following the common practice \cite{Saber2004,Zhou2009,Huang2010,Li2009,Li2010,Touri2009} we focus on balanced topologies on system (\ref{model0})-(\ref{model1}).

{\bf (A2)} The topology $\mathcal{G}(t)$ is balanced for all $t\geq 1$.

According to the definition of balancedness,  if $\{\mathcal{G}(t)\}$ is  undirected and the weight matrix $\mathcal{A}(t)$ is symmetric for all $t\geq 1$,
then $\{\mathcal{G}(t)\}$ are all balanced.

Under (A1) and (A2) we have the following lemma, whose proof is postponed to Appendix \ref{App_lemmas}.
\begin{lemma}\label{Lemma_3} Suppose that (A1) and (A2) are satisfied.
Let $z(t)=\Phi(t,i+1)z(i)$ for $t>i$.  If $a(t)\in (0,1/d_{\max})$  then
\begin{eqnarray*}
V\left(z(t)\right)\leq V\left(z(i)\right) \prod_{l=k^i}^{\widetilde{k}^t-1}\left(1-\frac{\delta_l(1-\delta_l)^2 \varepsilon_l }{n(n-1)^2} \right),
\end{eqnarray*}
where $\delta_l=\min_{t_{l}\leq t <t_{l+1} }a(t)$ and $\varepsilon_l=\min_{t_{l}\leq t <t_{l+1} }(1-a(t) d_{\max})$.
\end{lemma}

We also characterize robustness of protocol (\ref{model0})-(\ref{model1}) with respect to noise.
 This will be accomplished by accommodating a large class of noises as specified below.

For any random variables $X$ and $Y$, let Corr$(X,Y):=\frac{EXY-EX EY}{\sqrt{\mbox{Var}X\mbox{Var}Y}}$ denote
 the linear correlation coefficient between $X$ and $Y$. Following \cite{Wu2008}, we employ
  the notion of \emph{$\widetilde{\rho}$-mixing sequences} of  random variables.
Let $\{X_i\}_{i\geq 1}$ be a random variable sequence. For any subset $S, T \subset \mathds{N}$,  the sub $\sigma$-algebra  $\mathcal{F}_S:=\sigma(X_i,i\in S)$
 and
 \begin{eqnarray*}
\rho\left(\mathcal{F}_S, \mathcal{F}_T\right):=\sup \left\{\mbox{Corr}(X,Y): X\in L_2(\mathcal{F}_S), Y\in L_2(\mathcal{F}_T)\right\}.
\end{eqnarray*}
Define the $\widetilde{\rho}$-mixing coefficients by
\begin{eqnarray*}
\begin{aligned}
&\widetilde{\rho}(m):=\sup \Big\{\rho\left(\mathcal{F}_S, \mathcal{F}_T\right): \\
&~~~~\mbox{finite sets } S,T\subset \mathds{N} \mbox{ such that} \min_{i\in S, j\in T}|i-j|\geq m \Big\}
\end{aligned}
\end{eqnarray*}
for any integer $m\geq 0$. By definition,
$0\leq \widetilde{\rho}(m+1) \leq \widetilde{\rho}(m) \leq 1$ for all $m\geq 0$, and $\widetilde{\rho}(0)=1$ except for the special case when  all  $X_i$ are degenerate.

\begin{definition}
A
sequence of random variables $\{X_i\}_{i\geq 1}$  is said to be a $\widetilde{\rho}$-mixing sequence if there exists an integer $m>0$
such that $\widetilde{\rho}(m)<1$.
\end{definition}

Under this definition, we give the following assumption for protocol (\ref{model0})-(\ref{model1}).

\vskip 1mm
{\bf (A3)}
For any network topology sequence $\{\mathcal{G}(t)\}_{t\geq 1}$,
the noise sequence $\{w_{ji}(t)\}_{t\geq 1, i=1,\ldots,n, j\in\mathcal{N}_i(t)}$ is a zero-mean
$\widetilde{\rho}$-mixing sequence satisfying $v:=\sup_{i,j,t} \mbox{Var}(w_{ji}(t))<\infty$.
\vskip 1mm

\begin{remark}\label{mix_rem}
It is well known that $\widetilde{\rho}$-mixing noises include as special cases $\phi$-mixing noises \cite{Wang2013},  i.i.d. noises and
martingale difference noises, see \cite{Shao1993}.
\end{remark}

A basic property of $\widetilde{\rho}$-mixing sequences is cited here.

\begin{lemma}[Theorem 2.1 in \cite{Utev2003}]\label{Lemma_4}
Suppose that for an integer $m\geq 1$ and a real number $0\leq r < 1$,
$\{X_i\}_{i\geq 1}$ is a sequence of random variables with $\widetilde{\rho}(m)\leq r$, with $EX_i=0$ and
$E|X_i|^2<\infty$ for every $i\geq 1$. Then there is a positive constant $D=D(m, r)$ such that for all $k\geq 1$,
$$E|\sum_{i=1}^k X_i |^2 \leq D \sum_{i=1}^k E|X_i|^2.$$
\end{lemma}

Before the statement of the main result of this section, we give the following lemma first.

\begin{lemma}\label{Lemma_rate1}
Suppose that (A1) is satisfied with $\delta\leq 1/2$. Then for any constant $c_1>0$ and integer $t^*\geq 0$,
\begin{eqnarray}\label{Lemma_r1_00}
\prod_{j=k^{i}}^{\widetilde{k}^t-1}  \left(1-\frac{c_1}{t_{j+1}^{1-\delta}+t^*} \right)<\left(\frac{i^{1-\delta}+2c+t^*}{(t+1)^{1-\delta}+t^*}\right)^{\frac{c_1}{2c}},
\end{eqnarray}
and
\begin{eqnarray}\label{Lemma_r1_01}
&&\prod_{j=k^{i}}^{\widetilde{k}^t-1}  \left(1-\frac{c_1}{(t_{j+1}^{1-\delta}+t^*)(\log t_{j+1}+t^*)} \right)\nonumber\\
&& <\Big(\frac{\log[2c+i^{1-\delta}+t^*]}{\log[(t+1)^{1-\delta}+t^*]}\Big)^{\frac{c_1(1-\delta)}{2c}},
\end{eqnarray}
where $c$ is the same constant appearing in (A1).
\end{lemma}

The proof of Lemma \ref{Lemma_rate1} is
in Appendix A.
The following theorem presents  a sufficient condition for consensus.

\begin{theorem}\label{suff1}
Suppose that (A1) is satisfied with  $\delta\leq 1/2$, and
(A2) and (A3) hold. Then for any initial state $x(1)$,
there exists an open-loop control of the gain sequence $\{a(t)\}$ such that the system (\ref{model0})-(\ref{model1}) reaches unbiased mean square average-consensus, with a convergence rate $E[V(x(t))]=O(1/t^{1-2\delta})$ if $\delta<1/2$, and $E[V(x(t))]=O(1/\log t)$
if $\delta=1/2$.
\end{theorem}

\begin{proof}
Case I: $\delta<1/2$.
Choose $a(t)=\frac{\alpha}{t^{1-\delta}+t^*}$ with
\begin{eqnarray}\label{crit3_1}
\alpha\geq \frac{32n(n-1)^4 c}{(2n-3)^2}~~\mbox{and}~~t^*\geq \lfloor 2\alpha (n-1) a_{\max}\rfloor.
\end{eqnarray}
First we get $a(t)\leq \frac{\alpha}{ 2\alpha (n-1) a_{\max}} \leq \frac{1}{2d_{\max}}$, which indicates
$I-a(t) L(t)$ is a nonnegative matrix for any $t\geq 1$.
Recall that
$$\Phi(t,i)=\left[I-a(t) L(t)\right]\cdots \left[I-a(i) L(i)\right]$$
and $\prod_{i=j}^t (\cdot)=I$ for $t<j$.
Using (\ref{model3}) repeatedly, we get
\begin{eqnarray}\label{the_1_2}
\begin{aligned}
x(t+1)=\Phi(t,1)x(1)+\sum_{i=1}^{t}a(i)\Phi(t,i+1)\widehat{w}(i).
\end{aligned}
\end{eqnarray}
Take $\pi=(\frac{1}{n},\ldots,\frac{1}{n})\in\mathbb{R}^n$. By (A2) we have $\pi L(t)=0$, and hence $\pi \Phi(t,i)=\pi$.
 Take
\begin{eqnarray}\label{the_yi_def}
Y(i)=a(i)\left[\Phi(t,i+1)\widehat{w}(i)- \left(\pi\widehat{w}(i)\right)\mathds{1}  \right]\in\mathbb{R}^n.
\end{eqnarray}
Then by (\ref{the_1_2}), we have
\begin{eqnarray}\label{the_1_2_1}
&&x(t+1)-x_{\rm{ave}}(t+1)\mathds{1}=x(t+1)-(\pi x(t+1))\mathds{1}\nonumber\\
&&=\Phi(t,1)x(1)-\left(\pi x(1)\right)\mathds{1}\nonumber\\
&&~~~~ +\sum_{i=1}^{t}a(i)\left[\Phi(t,i+1)\widehat{w}(i)-(\pi \widehat{w}(i))\mathds{1}\right]\nonumber\\
&&=\Phi(t,1)x(1)-\left(\pi \Phi(t,1) x(1)\right)\mathds{1} +\sum_{i=1}^{t}Y(i).
\end{eqnarray}
Because $V(x)=\|x-(\pi x)\mathds{1}\|^2$,  by (A3) and (\ref{the_1_2_1}) we have
\begin{eqnarray}\label{the_1_2_2}
\begin{aligned}
E[V(x(t+1))]
= V\left(\Phi(t,1)x(1) \right)+E\Big\|\sum_{i=1}^{t}Y(i)\Big\|^2.
\end{aligned}
\end{eqnarray}
With Lemma \ref{Lemma_4} we have
\begin{eqnarray}\label{the_1_7}
&&E\Big\|\sum_{i=1}^{t}Y(i)\Big\|^2=E\sum_{j=1}^n \left[\sum_{i=1}^{t}Y_j(i)\right]^2\nonumber\\
&&\leq O\bigg( \sum_{j=1}^n \sum_{i=1}^{t}E Y_j^2(i)\bigg)=O\bigg(\sum_{i=1}^{t}E\|Y(i)\|^2\bigg)\nonumber\\
&&\leq O\bigg( \sum_{i=1}^{t}a^2(i)E\left[V\left(\Phi(t,i+1)\widehat{w}(i)\right)\right]\bigg).
\end{eqnarray}
By Lemma \ref{Lemma_3}, (\ref{crit3_1}) and (\ref{Lemma_r1_00}) we have for any $x\in\mathbb{R}^n$,
\begin{eqnarray}\label{crit3_4}
&&E\left[V\left(\Phi(t,i+1)\widehat{w}(i)\right)\right]\nonumber\\
&&\leq E\left[V\left(\widehat{w}(i)\right)\right] \prod_{j=k^{i}}^{\widetilde{k}^t-1}  \bigg[1-\left(\frac{2n-3}{2(n-1)}\right)^2\frac{a(t_{j+1})}{2n(n-1)^2}  \bigg]\nonumber\\
&&\leq E\left[V\left(\widehat{w}(i)\right)\right]\prod_{j=k^{i}}^{\widetilde{k}^t-1}  \left(1-\frac{4c}{t_{j+1}^{1-\delta}+t^*}\right)\nonumber\\
&&< E\left[V\left(\widehat{w}(i)\right)\right]\left(\frac{2c+i^{1-\delta}+t^*}{(t+1)^{1-\delta}+t^*}\right)^2.
\end{eqnarray}
Because $E[V\left(\widehat{w}(i)\right)]$ is bounded, from (\ref{the_1_7}) and (\ref{crit3_4}) we get
\begin{eqnarray}\label{crit3_7}
&&E\Big\|\sum_{i=1}^{t}Y(i)\Big\|^2= O\bigg(\frac{1}{[(t+1)^{1-\delta}+t^*]^{2}} \nonumber\\
&&~~\cdot \sum_{i=1}^{t}\frac{(2c+i^{1-\delta}+t^*)^{2}}{(i^{1-\delta}+t^*)^2}\bigg)=O\left(\frac{1}{t^{1-2\delta}}\right).
\end{eqnarray}
Also, similar to (\ref{crit3_4}) we get $V(\Phi(t,1)x(1))=O(1/t^{2-2\delta})$, so
taking (\ref{crit3_7}) into (\ref{the_1_2_2}) yields  $E[V(x(t))]=O(1/t^{1-2\delta})$.

It remains to evaluate the limit of $x_{\rm{ave}}(t)$.
Let
 \begin{eqnarray}\label{the_1_19_b1}
\begin{aligned}
x^*=x_{\rm{ave}}(\infty)=\pi x(\infty)= \pi x(1)+  \sum_{i=1}^{\infty}a(i)\pi \widehat{w}(i).
 \end{aligned}
\end{eqnarray}
By (\ref{the_1_19_b1}) we
 obtain
\begin{eqnarray}\label{the_1_19}
Ex^*=\pi x(1)+  \sum_{i=1}^{\infty}a(i)\pi E\widehat{w}(i)=\pi x(1).
\end{eqnarray}
Also, by Lemma \ref{Lemma_4} we have
\begin{eqnarray}\label{the_1_20}
&&\mbox{Var}(x^*)=E\Big|\sum_{i=1}^{\infty}a(i)\pi \widehat{w}(i)\Big|^2\\
&&=O\Big(\sum_{i=1}^{\infty} a^2(i) \Big)=O\Big(\sum_{i=1}^{\infty} \frac{\alpha^2}{(t^{1-\delta}+t^*)^2} \Big)<\infty,\nonumber
\end{eqnarray}
so by Definition \ref{consensus_def} the system (\ref{model0})-(\ref{model1}) reaches unbiased mean square average-consensus.\\
Case II: $\delta=1/2$.
Choose $a(t)=\frac{\alpha}{(\sqrt{t}+t^*)\log(t+t^*)}$ with
\begin{eqnarray*}\label{crit3_1a}
\alpha\geq \frac{64n(n-1)^4 c}{(2n-3)^2}~~\mbox{and}~~t^*\geq \lfloor 2\alpha (n-1) a_{\max}\rfloor.
\end{eqnarray*}
With (\ref{Lemma_r1_01}), similar to (\ref{crit3_4}) we get
\begin{eqnarray*}\label{crit3_4a}
&&E\left[V\left(\Phi(t,i+1)\widehat{w}(i)\right)\right]\nonumber\\
&&< E\left[V\left(\widehat{w}(i)\right)\right]\Big(\frac{\log[2c+\sqrt{i}+t^*]}{\log[\sqrt{t+1}+t^*]}\Big)^2,
\end{eqnarray*}
so similar to (\ref{crit3_7}) we have
\begin{eqnarray}\label{crit3_7a}
&&E\Big\|\sum_{i=1}^{t}Y(i)\Big\|^2= O\bigg(\frac{1}{\log^2[\sqrt{t+1}+t^*]} \\
&&~~\cdot \sum_{i=1}^{t}\frac{\log^2[2c+\sqrt{i}+t^*]}{(\sqrt{i}+t^*)^2 \log^2(i+t^*)}\bigg)=O\left(\frac{1}{\log t}\right).\nonumber
\end{eqnarray}
Similar to Case I, taking (\ref{crit3_7a}) into (\ref{the_1_2_2}) yields  $E[V(x(t))]=O(1/\log t)$. Also, because $\sum_{l=2}^{\infty} \frac{1}{l \log^2 l}<\infty$ (Subsection 1.3.9 in \cite{Zwillinger2003}) we get $\mbox{Var}(x^*)=O(\sum_{i=1}^{\infty} a^2(i))<\infty$. With the same discussion as Case I
the system (\ref{model0})-(\ref{model1}) reaches unbiased mean square average-consensus.
\end{proof}




\subsection{Necessary Conditions for Consensus}\label{nece_sec}

Let $\mathcal{G}_1=(\mathcal{V},\mathcal{E}_1)$ be an undirected complete graph, which means each vertex can receive the information of all the others. Let $\mathcal{G}_2=(\mathcal{V},\mathcal{E}_2)$ be the graph which has only one undirected edge between vertexes $1$ and $2$ without any other edges.
For both $\mathcal{G}_1$ and $\mathcal{G}_2$, we assume the weights of their edges are all equal to $1$.
Thus, the corresponding Laplacian matrices for  $\mathcal{G}_1$ and  $\mathcal{G}_2$ are $L_1=n I-\mathds{1}\mathds{1}'
\in \mathds{R}^{n\times n}$ and
\begin{eqnarray*}\label{nece1_3}
L_2=\left(
\begin{array}{ccccc}
1 & -1 & 0 & \cdots & 0 \\
-1 & 1 & 0 & \cdots & 0 \\
0 & 0 & 0 & \cdots & 0 \\
\cdots & \cdots & \cdots & \ddots & \cdots \\
0 & 0 & 0 & \cdots &0
\end{array}
\right)\in \mathds{R}^{n\times n}.
\end{eqnarray*}
respectively.
Define $v_1:=n^{-1/2} \mathds{1}\in\mathds{R}^n$, $v_2:=\frac{1}{\sqrt{2}}(1,-1,0,\ldots,0)'\in\mathds{R}^n$, and
$$v_i:=(i^2-i)^{-1/2}(\underbrace{1,\ldots,1}_{i-1},1-i,0,\ldots,0)' \in\mathds{R}^n$$
for $i\in [3,n]$.
It is easy to compute that: $v_i'v_j=0$ for $i\neq j$; $L_1 v_1=\mathbf{0}$; $L_1 v_i=n v_i$ for $i\geq 2$; $L_2 v_2=2v_2$ and $L_2 v_i=\mathbf{0}$ for $i\neq 2$. From this we have the following proposition:

\begin{proposition}\label{prop_matrix}
Let $P:=(v_1,v_2,\ldots,v_n)\in\mathds{R}^{n\times n}$.
Then $P'P=I$,
$P\mbox{diag}(0,n,n,\ldots,n)P'=L_1$, and $P\mbox{diag}(0,2,0,\ldots,0)P'=L_2$.
\end{proposition}

The necessary condition says for any gain sequence, the consensus may  not be reached when the extensible exponent $\delta$ is larger than $1/2$ under the following noise condition:

{\bf (A4)}
For any network topology sequence $\{\mathcal{G}(t)\}_{t\geq 1}$,
assume the noises $\{w_{ji}(t)\}$ are zero-mean random variables satisfying: i)
$E[w_{j_1 i_1}(t_1)w_{j_2 i_2}(t_2)]=0$ for any $t_1\neq t_2$, $(j_1,i_1)\in\mathcal{E}(t_1)$, and $(j_2,i_2)\in\mathcal{E}(t_2)$;
ii) there exist constants $0<\underline{v}\leq \overline{v}$ such that for any non-empty edge set $\mathcal{E}(t)$ and real numbers $\{c_{ji}\}$,
\begin{eqnarray*}\label{A4_00}
\underline{v} \sum_{(j,i)\in\mathcal{E}(t)} c_{ji}^2 \leq E\bigg[\sum_{(j,i)\in\mathcal{E}(t)} c_{ji} w_{ji}(t)  \bigg]^2\leq \overline{v} \sum_{(j,i)\in\mathcal{E}(t)} c_{ji}^2.
\end{eqnarray*}

\begin{theorem}\label{nece1}
Suppose the noise satisfies (A4). Then for any constant $\delta^*>1/2$,  any non-consensus
initial state, and any gain sequence $\{a(t)\}_{t\geq 1}$,
there exists at least one topology sequence
$\{\mathcal{G}(t)\}_{t\geq 1}=\{(\mathcal{V},\mathcal{E}(t),\mathcal{A}(t))\}_{t\geq 1}$ satisfying (A1)-(A2) with $\delta=\delta^*$,  such that system (\ref{model0})-(\ref{model1}) cannot reach mean square consensus.
\end{theorem}

\begin{proof}
The main idea of this proof is: Choose $t_k=t_{k-1}+c \lfloor t_{k-1}^{\delta}\rfloor$. Let $\{a(t)\}_{t\geq 1}$ be an arbitrary gain sequence.
For any $k\geq 1$ and $t_k\leq t<t_{k+1}$, select $\mathcal{G}(t)$ to be $\mathcal{G}_1$ if $a(t)$ is the minimal value of $\{a(s),t_k\leq s<t_{k+1}\}$, and to be
$\mathcal{G}_2$ otherwise. It can be verified that our choice satisfies both (A1) and (A2). With Proposition \ref{prop_matrix}, we
conclude that  system
(\ref{model0})-(\ref{model1}) cannot reach  consensus in mean square. The detailed proof is in Appendix \ref{App_nece1}.
\end{proof}

\subsection{A Critical Condition for Consensus}

The consensus conditions of the first-order average-consensus protocols with deterministic topologies and additive noises
have been investigated recently. However, the
best known condition on topology for consensus to date is the uniform joint-connectivity \cite{Touri2009,Li2010,Huang2012}. On the other hand, if this type of protocols contains no noise, they can reach consensus under a much relaxed infinite joint-connectivity condition \cite{Touri2011}. There exists a huge gap between these two consensus conditions. This paper proposes an extensible joint-connectivity condition which is an intermediate condition between the uniform joint-connectivity and infinite joint-connectivity.
Under our new condition we investigate a basic problem:  what is the critical extensible exponent under which we can control the system to reach consensus?
Note that there does not exist a central controller who knows the global information during the system's evolution, and
the consensus control is defined by designing off-line gains $a(t)$ such that all the agents achieve the same final state.

Let $\mathscr{G}$ be the set of topology sequences satisfying (A1) and (A2),
and $\mathscr{W}$ be the set of  noise sequences satisfying (A3).
\begin{theorem}\label{ce_1}
If and only if $\delta\leq 1/2$,  where $\delta$ is the extensible exponent appearing in (A1),
there exists an open-loop control of the gain sequence $\{a(t)\}$ such that system (\ref{model0})-(\ref{model1}) reaches unbiased mean square average-consensus for any  topology sequence $\{\mathcal{G}(t)\}\in\mathscr{G}$, any noise sequence $\{w_{ji}(t)\}\in\mathscr{W}$,  and any non-consensus initial state.
\end{theorem}
\begin{proof}
This follows immediately from Theorems \ref{suff1} and \ref{nece1},
since any noise satisfying (A4) must satisfy (A3).
\end{proof}

\begin{remark}\label{rem_A2_a}
The balancedness of network topologies can guarantee that the expectation of the final consensus value is equal to the average value of the initial states $\frac{1}{n}\sum_{i=1}^n x_i(1)$.
In addition, by (\ref{the_1_20}), the variance of the consensus value can be arbitrarily small if we choose $t^*$ to be large enough.
Overall, we can control the final consensus value to be arbitrarily close to the average value $\frac{1}{n}\sum_{i=1}^n x_i(1)$.
\end{remark}


\begin{remark}\label{rem_A2}
Without  assumption (A2), Theorem \ref{ce_1} should still hold if one replaces  \emph{unbiased mean square average-consensus} by  \emph{mean square consensus}.
However, its proof is quite difficult because it is related
to a well-known conjecture in the field of probability that the convergence rate of a general inhomogeneous Markov chain is a negative exponential function. This conjecture was formulated as Problem 1.1 in \cite{Saloff2009}.
We remark that the reference \cite{Huang2012} obtained the convergence under
the uniformly joint-connectivity by the classical infinitesimal analysis which cannot be used to obtain convergence rates or analyze the critical connectivity condition.
Currently, almost all papers concerning convergence speeds of the distributed consensus protocol with time-varying topologies assume that the topologies are undirected, or balanced, or have a common stationary distribution \cite{Saber2004,Zhou2009,Huang2010,Li2009,Li2010,Touri2009,Touri2011,Patterson2010,Yin2013}.
\end{remark}

\section{Fastest Convergence Rates of Consensus}
\label{Convergence_speed}

This section establishes bounds on
the fastest convergence  rate  to the unbiased mean square average-consensus among all gain functions
under unknown switching topologies.
Different from the noise-free systems \cite{Xiao2004,Kim2006,Olshevsky2011},
$x^*$ in Definition \ref{consensus_def} is a random variable whose value is uncertain. Also, if system (\ref{model0})-(\ref{model1}) reaches consensus in mean square, it must be true that
 $\lim_{t\rightarrow\infty}E[V(x(t))]=0$, so we use $E[V(x(t))]$ to  measure the convergence rate to consensus instead of $E\|x(t)-x^*\mathds{1}\|^2$.
In this paper the \emph{fastest convergence rate} of consensus at time $t$ is the minimal value of $E[V(x(t))]$ among all controls
$a(1)\geq 0, a(2)\geq 0, \ldots, a(t-1)\geq 0$ for a topology sequence $\mathcal{G}(1), \ldots, \mathcal{G}(t-1)$.
This rate depends on the time-varying topologies, however our protocol assumes each node only knows its local information
and the global topology information is unknown. As a result, its exact value cannot be obtained.
A simplified notion of convergence rate will be first defined. Let
 \begin{eqnarray}\label{rho1}
\rho_1(t):=\inf_{a(1)\geq 0,\ldots,a(t-1)\geq 0}\inf_{\mathcal{G}(1),\ldots,\mathcal{G}(t-1)} E[V(x(t))]
 \end{eqnarray}
 be the fastest convergence rate under the best topologies.
 Here we recall that $\{\mathcal{G}(t)\}_{t\geq 1}=\{(\mathcal{V},\mathcal{E}(t),\mathcal{A}(t))\}_{t\geq 1}$ is the topology sequence and note that $\rho_1(t)$ depends on  the noises and the initial state $x(1)$.

 We also consider the fastest convergence rate under the worst topologies.
Define $\mathscr{G}_{\delta,c}$
as the set of topology sequences satisfying (A1)-(A2), where $\delta,c,$ are the constants appearing in (A1).
Let
 \begin{eqnarray}\label{rho2}
\rho_2(t):=\inf_{a(1)\geq 0,\ldots,a(t-1)\geq 0}\sup_{\{\mathcal{G}(k)\}\in \mathscr{G}_{\delta,c}}E[V(x(t))]
 \end{eqnarray}
denote the fastest convergence rate with respect to the worst topologies satisfying (A1)-(A2).
We note that $\rho_2(t)$ depends on $\delta,c$, the noises and the initial state.

By the definitions of $\rho_1(t)$ and $\rho_2(t)$ we have for any topology sequences satisfying (A1)-(A2), its corresponding fastest convergence rate will be neither faster than $\rho_1(t)$ nor slower than
$\rho_2(t)$ provided that the initial state and noises are same.
Theorem \ref{suff1} gives a upper bound for $\rho_2(t)$, and in the following subsection we will consider
the lower bounds for $\rho_1(t)$ and $\rho_2(t)$.

\subsection{Lower Bounds}\label{lower bounds}

In this subsection we will give lower bounds on $\rho_1(t)$ and $\rho_2(t)$, respectively under (A4). The lower bounds on the fastest convergence rate indicate that for any control the convergence rate will not be faster than them.
Before the estimation of $\rho_1(t)$  we need introduce the following lemma:

\begin{lemma}\label{Lemma_lower}
Let $L\in\mathds{R}^{n\times n}$ be the Laplacian matrix of any weighted directed graph.
Then for any $x\in\mathds{R}^n$ and constant $a>0$ we have
$V[(I-aL)x]\geq \left(1-a\lambda_{\max}(L+L')\right) V(x),$
where $\lambda_{\max}(\cdot)$ denotes the largest eigenvalue.
\end{lemma}

The proof of this lemma  is
in Appendix A.

The following theorem gives a lower bound of $\rho_1(t)$.
\begin{theorem}\label{crit_anytop}
Suppose that the noises satisfy (A4).
Then for any
non-consensus initial state, under protocol (\ref{model0})-(\ref{model1}) there exists a constant $c'>0$ such that
$\rho_1(t)\geq c'/t$ for all $t\geq 1$.
\end{theorem}
\begin{proof}
For any $t>1$,  we only need to consider the case of
 $\mathcal{E}(k)$ is not empty for all $1\leq k\leq t$, since if $\mathcal{E}(k)$ is empty then $x(k+1)=x(k)$, which results in the waste of the time step.

First, because $\widehat{w}_i(k)=\sum_{j\in\mathcal{N}_i(k)} a_{ij}^k w_{ji}(k)$ with
$a_{ij}^k\geq 1$, by (A4) there exists a constant $c_1>0$ such that
\begin{eqnarray}\label{cap_1}
\begin{aligned}
E\left[V\left(\widehat{w}(k)\right)\right]\geq c_1.
\end{aligned}
\end{eqnarray}
Also, by Gershgorin's circle theorem we have
\begin{eqnarray}\label{cap_3_1}
\begin{aligned}
&\lambda_{\max}(L(k)+L'(k))\\
&~\leq \max_{1\leq i\leq n}\Big(2 L_{ii}(k)+ \sum_{j\neq i} |L_{ji}(k)+L_{ij}(k)| \Big)\\
&~\leq 4(n-1)a_{\max}:=c_2,~~~~\forall k\geq 1.
\end{aligned}
\end{eqnarray}
 Let $t^*\in [1,t+1]$ be the minimum time such that if $k\geq t^*$ then $a(k)<1/c_2$.
By (\ref{the_1_2_2}) and (A4) it can be computed that
\begin{eqnarray}\label{cap_3_0}
\begin{aligned}
&E\left[V(x(t+1))\right]=E\left[V(\Phi(t,t^*)x(t^*))\right]\\
&~~+\sum_{i=t^*}^{t}a^2(i)E\left[V\left(\Phi(t,i+1)\widehat{w}(i)\right)\right].
\end{aligned}
\end{eqnarray}

If $t^*=t+1$, we have $a(t)\geq  1/c_2$. Then by (\ref{cap_3_0}) and (\ref{cap_1}),
\begin{eqnarray*}\label{cap_2}
\begin{aligned}
E\left[V(x(t+1))\right]&\geq a^2(t)E\left[V\left(\widehat{w}(t)\right)\right]\geq c_1/c_2^2
\end{aligned}
\end{eqnarray*}
and the result follows.

Hence,
we only need to consider the case  $t^*\leq t$.
By (\ref{cap_3_1}) and repeatedly using Lemma \ref{Lemma_lower} we have
\begin{eqnarray*}
\begin{aligned}
E\left[V(\Phi(t,t^*)x)\right]\geq  E\left[V(x)\right] \prod_{j=t^*}^t\left(1-c_2 a(j) \right).
\end{aligned}
\end{eqnarray*}
Taking this into (\ref{cap_3_0}) yields
\begin{eqnarray}\label{cap_3}
\begin{aligned}
&E\left[V(x(t+1))\right]\geq E\left[V(x(t^*))\right] \prod_{j=t^*}^t\left(1-c_2 a(j) \right)\\
&~~+\sum_{i=t^*}^{t}a^2(i)E\left[V\left(\widehat{w}(i)\right)\right]\prod_{j=i+1}^t\left(1-c_2 a(j) \right).\\
\end{aligned}
\end{eqnarray}
Let $I_{\{\cdot\}}$ be the indicator function. Since
\begin{eqnarray*}\label{cap_4}
\begin{aligned}
1-c_2 a(j)&=\left[1-c_2a(j) I_{\{a(j)>\frac{1}{c_2 t}\}}\right]\\
&~~\cdot \left[1-c_2 a(j) I_{\{a(j)\leq \frac{1}{c_2 t}\}}\right],
\end{aligned}
\end{eqnarray*}
and
\begin{eqnarray*}\label{cap_5}
\begin{aligned}
\prod_{j=1}^t\left(1-c_2 a(j) I_{\{a(j)\leq \frac{1}{c_2 t}\}} \right)\geq \left(1-\frac{1}{t}\right)^t\geq \frac{1}{e},
\end{aligned}
\end{eqnarray*}
 from (\ref{cap_3}) and (\ref{cap_1}), we
 obtain
\begin{eqnarray}\label{cap_6}
\begin{aligned}
&E\left[V(x(t+1))\right]\\
&\geq \frac{1}{e}E\left[V(x(t^*))\right] \prod_{j=t^*}^t\left(1-c_2 a(j) I_{\{a(j)> \frac{1}{c_2 t}\}}\right)\\
&\ \ +\frac{c_1}{e}\sum_{i=t^*}^{t}a^2(i)\prod_{j=i+1}^t\left(1-c_2 a(j)I_{\{a(j)>\frac{1}{c_2 t}\}} \right).
\end{aligned}
\end{eqnarray}

It remains to discuss the value of the right side of (\ref{cap_6}). We first consider $E[V(x(t^*))]$.
If $t^*=1$ then $E[V(x(t^*))]=V(x(1))$. Otherwise, by the definition of $t^*$ we have $a(t^*-1)\geq 1/c_2$.
As a result, similar to (\ref{cap_3_0}) and by (\ref{cap_1}) we have
$$E\left[V(x(t^*))\right]\geq a^2(t^*-1)E\left[V\left(\widehat{w}(t^*-1)\right)\right]\geq c_1/c_2^2.$$
These lead to
\begin{eqnarray}\label{cap_7}
\begin{aligned}
E\left[V(x(t^*))\right]\geq \min\left\{V(x(1)), c_1/c_2^2  \right\}.
\end{aligned}
\end{eqnarray}
Also, from $y_1\geq y_2 I_{\{y_1>y_2\}}$ for any $y_1,y_2\geq 0$, we get
\begin{eqnarray}\label{cap_8}
&&\sum_{i=t^*}^{t}a^2(i)\prod_{j=i+1}^t\left(1-c_2 a(j)I_{\{a(j)>\frac{1}{c_2 t}\}} \right)\nonumber\\
&& \geq \sum_{i=t^*}^{t}\frac{a(i)}{c_2 t}I_{\{a(i)>\frac{1}{c_2 t}\}}\prod_{j=i+1}^t\left(1-c_2 a(j)I_{\{a(j)>\frac{1}{c_2 t}\}} \right)\nonumber\\
&&=\frac{1}{c_2^2t}\bigg(1-\prod_{j=t^*}^t\left(1-c_2 a(j) I_{\{a(j)>\frac{1}{c_2 t}\}} \right)   \bigg),
\end{eqnarray}
where the last line uses the classical equality
\begin{eqnarray*}
\sum_{i=1}^{t} b_i \prod_{j=i+1}^{t} (1-b_j)=1-\prod_{i=1}^{t} (1-b_i), \forall b_i\in\mathds{R}, i\geq 1,
\end{eqnarray*}
which can be obtained by induction. Here we recall that $\prod_{j=i}^{t}(\cdot):=1$ if $t<i$.
Take $z_t=\prod_{j=t^*}^t (1-c_2 a(j) I_{\{a(j)>\frac{1}{c_2 t}\}} )$.
By substituting (\ref{cap_8}) into (\ref{cap_6}) we have
\begin{eqnarray*}\label{cap_9}
\begin{aligned}
E\left[V(x(t+1))\right]&\geq\frac{1}{e}E\left[V(x(t^*))\right]z_t+\frac{c_1}{e c_2^2 t}(1-z_t)\\
&\geq \min\left\{\frac{1}{e}E\left[V(x(t^*))\right], \frac{c_1}{e c_2^2 t}\right\}.
\end{aligned}
\end{eqnarray*}
From this and (\ref{cap_7}) our result is obtained.
\end{proof}

For $\rho_2(t)$ we get the following lower bound, whose
proof
is in Appendix \ref{App_crit2}.

\begin{theorem}\label{crit2}
Assume (A1) is satisfied with $\delta<1/2$, and (A2) and (A4) hold.
Then for any inconsistent initial state, under protocol
(\ref{model0})-(\ref{model1}) there exists a constant $c'\in(0,1)$ such that
$\rho_2(t)\geq  c'/t^{1-2\delta}.$
\end{theorem}

\subsection{Fastest Convergence Rates and Sub-optimal Open-loop Control}

The convergence speed is one of the most important performances of
distributed consensus algorithms for networked systems.
Most existing work
focuses on noise-free algorithms  \cite{Saber2004,Xiao2004,Olshevsky2011,Angeli2008,Angeli2009,Hatano2005,Akar2008,Zhou2009}
where the control gains $\{a(t)\}$ are constant.
Among these,
some try to
maximize the convergence speed by optimizing weighted network topologies \cite{Xiao2004,Kim2006}.
There are some results considering convergence speed of distributed consensus algorithms with fixed topologies and additive noises
\cite{Yin2011,Yin2013,Wang2013}. Nevertheless, it appears that our paper is the first to optimize the convergence rate of this type of protocols with time-varying network topologies and additive noises. It is noted that in our system each node only knows its own and neighbors' information and the network topologies cannot be real-time controlled.


In this paper the fastest convergence rate of consensus at time $t$ is the minimal value of $E[V(x(t))]$ among all the gain functions
$a(1)\geq 0, a(2)\geq 0, \ldots, a(t-1)\geq 0$ which are the only controllable variables.
 Recall that $\rho_1(t)$  defined by (\ref{rho1}) denotes the fastest convergence rate for the best topologies,
and $\rho_2(t)$ defined by (\ref{rho2}) denotes the fastest convergence rate for the worst topologies satisfying (A1)-(A2).
With the same noise sequence it is clear that $\rho_1(t)\leq \rho_2(t)$ from their definitions.

\begin{theorem}\label{rho1_est}
Suppose that the noise satisfies (A4) and the initial state is not in consensus, then $\rho_1(t)=\Theta\left(\frac{1}{t}\right)$ under system (\ref{model0})-(\ref{model1}).
\end{theorem}

\begin{proof}
By the definitions of $\rho_1(t)$ and $\rho_2(t)$ we have $\rho_1(t)\leq \rho_2(t)$ with $\delta=0$.
By Theorem \ref{suff1}
we have $\rho_1(t)=O(\frac{1}{t})$. Combing this with Theorem \ref{crit_anytop} yields our result.
\end{proof}


\begin{remark}\label{rem_conject1}
We just evaluate the fastest convergence rate to the accurate order. In fact,
it is conjectured that  $\rho_1(t)=\frac{b_1}{t}(1+o(1))$ under (A4),  where $b_1$ is a constant depending on $n$, $a_{\max}$ and $v$ only.
\end{remark}

\begin{theorem}\label{rho2_est}
Suppose that the topology sequence  $\{\mathcal{G}(t)\}$ satisfies (A1)-(A2) with $\delta<\frac{1}{2}$, the noises satisfy (A4),  and the initial state is not in consensus.
Then under system (\ref{model0})-(\ref{model1}),\\
i) $\rho_2(t)=\Theta\left(\frac{1}{t^{1-2\delta}}\right)$.\\
ii) The unbiased mean square average-consensus will be reached with a rate  $O(\frac{1}{t^{1-2\delta}})$ by
 choosing $a(t)=\frac{\alpha}{t^{1-\delta}+t^*}$, where $\alpha$ and $t^*$ are two constants satisfying (\ref{crit3_1}).
\end{theorem}
\begin{proof}
Since any noise satisfying (A4) must satisfy (A3), (i) follows immediately from Theorems \ref{crit2} and \ref{suff1}, and (ii) follows immediately  from the proof of Theorem \ref{suff1}.
\end{proof}

From Theorems \ref{rho1_est} and \ref{rho2_est} we get the following corollary.

\begin{corollary}
Suppose that the noise satisfies (A4) and the initial state is not in consensus.
 Then for system (\ref{model0})-(\ref{model1}) with any balanced and uniformly jointly connected topology sequence,
 i) the fastest convergence rate $\inf_{a(1)\geq 0,\ldots,a(t-1)\geq 0} E[V(x(t))]$  is $\Theta\left(\frac{1}{t}\right)$;
ii)  the unbiased mean square average-consensus will be reached with a rate $\Theta\left(\frac{1}{t}\right)$ by choosing $a(t)=\frac{\alpha}{t+t^*}$, where $\alpha$ and $t^*$ are two constants satisfying (\ref{crit3_1}).
\end{corollary}

\begin{proof}
i)  Let $\rho_3(t):=\inf_{a(1)\geq 0,\ldots,a(t-1)\geq 0} E[V(x(t))]$ which depends on the topology sequence.
 First by (\ref{rho1}) and Theorem \ref{rho1_est} we can get $\rho_3(t)\geq\rho_1(t)=\Theta(\frac{1}{t})$. Also, because the balanced and uniformly jointly connected topology condition equals to the condition (A1)-(A2) with $\delta=0$, by (\ref{rho2}) and  Theorem \ref{rho2_est} i) we can get $\rho_3(t)\leq\rho_2(t)=\Theta(\frac{1}{t})$. Thus, we have $\rho_3(t)=\Theta(\frac{1}{t})$. \\
ii) It follows immediately  from i) and Theorem \ref{rho2_est} ii).
\end{proof}

\section{Consensus under Non Stationary and Strongly Correlated Random Topologies}\label{sec_stoc}
As mentioned in Remark \ref{rem_adv},  one advantage of the extensible joint-connectivity is that it can be used to analyze systems with random topologies compared to the uniform joint-connectivity condition. This is because with probability $1$ there exists a finite time $T>0$ such that  random topologies satisfy (A1) for all $t\geq T$, even if
the topology processes are not stationary and strongly correlated.
In this section we consider random network topologies $\{\mathcal{G}(t)\}_{t\geq 1}=\{(\mathcal{V},\mathcal{E}(t),\mathcal{A}(t))\}_{t\geq 1}$ satisfying:

\textbf{(A1')} There exist three constants $K\in\mathbb{Z}^+$,  $\mu\in(0,1/2)$ and $p>0$ such that
for any $t\geq 1$,
\begin{eqnarray*}
\begin{aligned}
&P\Big(\bigcup_{t'=t}^{t+K-1}\mathcal{G}(t') \mbox{ is strongly connected}\\
&~~~~~~|\forall \mathcal{G}(i), 1\leq i<t \Big) \geq p  t^{-\mu}\log  t.
\end{aligned}
\end{eqnarray*}

\begin{remark}
 Assumption (A1')
includes
a wide class of non-stationary and strongly correlated random matrix sequence $\{\mathcal{G}(t)\}$, including as special cases the
 ergodic and stationary Markov processes used in \cite{Huang2010,Tahbaz2010,Yin2011}.
\end{remark}

\begin{remark}
The constant $\mu$ in (A1')  essentially corresponds to the extensible exponent $\delta$ in (A1).
Also, in practical applications this constant would be converted into a certain parameter of practical systems.
For example, in mobile ad-hoc networks,  because the probability of successful communication between two agents depends on their distance,
(A1') can be translated into a limitation to the growth rate of the distance between agents, where $\mu$ corresponds to a coefficient
of this growth rate, see the following (\ref{App_9})-(\ref{App_10}).
\end{remark}


From Theorem \ref{suff1}, we obtain the following result for the case of random network topologies, whose
proof
is contained in Appendix \ref{App_rand_topo}.
An application of Theorem \ref{rand_topo} is provided in the following subsection.

\begin{theorem}\label{rand_topo}
Consider the system given by (\ref{model0})-(\ref{model1}) with random network topologies satisfying (A1') and (A2). Assume that the noise satisfies (A3). Then from any initial state the system will reach unbiased mean square average-consensus with the convergence rate $E[V(x(t))]=O(\frac{1}{t^{1-2\mu}})$ if one selects $a(t)=\frac{\alpha}{t^{1-\mu}+t^*}$, where $\alpha$ and $t^*$ are two constants satisfying (\ref{crit3_1}), and $\mu$ is the same constant appearing in (A1').
\end{theorem}


\subsection{Application to Mobile Ad-Hoc Networks}\label{subsec_app}

To investigate the distributed consensus protocol with random networked topologies,
the existing results assume that the topologies are either i.i.d. or stationary Markov processes\cite{Hatano2005,Porfiri2007,Akar2008,Zhou2009,Tahbaz2010,Huang2010,Yin2011}. This assumption fits stationary
wireless networks; but in mobile systems
network topologies will no longer be
stationary because communications between nodes depend on their distances.
Different from the previous work, Theorem \ref{rand_topo} treats non-stationary random network topologies and can be applied to distributed computation of
mobile networks.  For example,
a mobile wireless sensor network or a multi-robot system needs to compute the average value of some data (such as temperature, humidity, light intensity, pressure
etc.) measured by each node (or agent).
Assume that the data of each agent $i$ are encoded to a scalar $x_i(0)$.
We aim to design a distributed protocol to obtain the average value of $x_i(0)$.
Since communication packet delivery ratios between agents depend on their distances, we must take into consideration of agent movement.

We adopt the first-order average-consensus protocol as our communication protocol. Let $x_i(t)$ be the state of agent $i$ at time $t$ which is initially set to be $x_i(0)$. We assume that each agent periodically, with period $T>0$, broadcasts its current state to all other agents.
To reduce signal interference the agents are arranged to send information in different times.
This leads to the following communication protocol.

In each period $[lT,(l+1)T)$, $l\geq 0$, at the time $lT+\frac{(i-1)T}{n}$, $1\leq i\leq n$, agent $i$ broadcasts its current state $x_i(lT+\frac{(i-1)T}{n})=x_i(lT):=x_i^l$  to
all other agents. The more bits its sending package contains, the more difficult this package is to be successfully received by other agents.
Practical wireless systems only transfer a finite number of bits per transmission.
Thus, the broadcasting signal is equal to $x_i^l$ with a quantization noise $\xi_i^l$.
Each link
$(i,j)$ has a probability which depends the distance between agents $i$ and $j$ for successfully receiving the information
\begin{eqnarray}\label{App_2}
\begin{aligned}
R_j^{i,l}=x_i^l+\xi_i^l+\zeta_j^{i,l},
\end{aligned}
\end{eqnarray}
where $\zeta_j^{i,l}$ denotes the possible reception error of agent $j$ that the error detection code in its received package cannot detect.
After agent $j$ receives the information $R_j^{i,l}$, it sends an acknowledgment to agent $i$ immediately. When the response signal reaches agent $i$,
agent $i$ can know the response signal coming from agent $j$ according to its carrier frequency if each agent is assigned a distinct carrier frequency. This process does not have to decode the signal
so we can assume that the response can be indeed received by agent $i$.
Agent $i$ collects the response only in the interval $(lT+\frac{(i-1)T}{n},lT+\frac{iT}{n})$.
Because the above sending-reception-response process contains no retransmission, its total time is very short and then agent $i$ can receive all responses to itself in $(lT+\frac{(i-1)T}{n},lT+\frac{iT}{n})$ if we choose $T$ to be a suitable real number. This fact indicates that each agent can collect all responses to itself, and will not wrongly collect the responses to others. Thus, each agent actually knows who receives its sending state.
Define
\begin{eqnarray*}\label{App_3}
\begin{aligned}
\mathcal{N}_i^l:=\big\{j:&\mbox{ Agents } i \mbox{ and } j \mbox{ receive each other's state}\\
&~\mbox{in the time interval } [lT,(l+1)T)\big\}.
\end{aligned}
\end{eqnarray*}
to be the neighbor set of agent $i$. At the end of every period $[lT,(l+1)T)$ each agent updates its state by
\begin{eqnarray}\label{App_4}
&&x_i^{l+1}:=x_i((l+1)T)= x_i^{l}+a(l)\sum_{j\in\mathcal{N}_i^l}( R_i^{j,l}-x_i^l)\nonumber\\
&&= x_i^{l}+a(l)\sum_{j\in\mathcal{N}_i^l}( x_j^l+\xi_j^l+\zeta_i^{j,l}-x_i^l)
 \end{eqnarray}
for all $l\geq 0$ and $1\leq i\leq n$, where the last line uses (\ref{App_2}).

 Let $\mathcal{G}^l:=(\mathcal{V},\mathcal{E}^l)$ where $\mathcal{E}^l:=\{(i,j):j\in\mathcal{N}_i^l\}$. Then $\mathcal{G}^l$ is an undirected graph.
 Also, it is natural to assume that $\{\xi_j^l+\zeta_i^{j,l}\}$ is a zero-mean and bounded variance $\tilde{\rho}$-mixing sequence. Thus, according to
Theorem \ref{rand_topo}, if the topologies $\{\mathcal{G}^l\}_{l\geq 1}$ satisfy
 (A1'), we can use the open-loop control $a(l)$ such that system (\ref{App_4}) reaches unbiased mean square average-consensus.

 Next we consider the movement restriction guaranteeing that the topologies $\{\mathcal{G}^l\}_{l\geq 1}$ satisfy
 (A1').
According to the log-normal shadowing model\cite{Rap2002,Wang2010}, the probability of agent $i$ successfully receiving a data packet from agent $j$ can be approximated by
 \begin{eqnarray}\label{App_5}
\begin{aligned}
&P_i^j(t)=\frac{1}{2}+\frac{1}{\sqrt{\pi}}\int_0^{\frac{S_j(t)-L_{i}^j(t)-R_{\rm{th}}}{\sqrt{2}\delta}}e^{-x^2}dx,
\end{aligned}
 \end{eqnarray}
 where $S_j(t)$ is the transmission signal strength of agent $j$ at time $t$, $L_{i}^j(t)$ is the path loss
 between agents $i$ and $j$ at time $t$, $R_{\rm{th}}$ is a constant depending on the size of the data package,
 and $\delta$ is the standard deviation of a Gaussian random variable ($S_j(t)$, $R_{\rm{th}}$ and $\delta$ are measured in dBm, $L_{i}^j(t)$ is measured in dB).

 For simplicity, we assume that the transmission signal strength $S_j(t)=S_j$ is a constant, and the path loss is estimated by the free-space path loss (FSPL) which is
  \begin{eqnarray*}\label{App_6}
\begin{aligned}
L_{i}^j(t)=32.45+20\log d_{ij}(t)+20\log f_j,
\end{aligned}
 \end{eqnarray*}
 where $d_{ij}(t)$ is the distance between agents $i$ and $j$ at time $t$ measured in kilometer, and $f_j$ is the carrier frequency of agent $j$ measured in MHz.
 Take $\alpha_j=\frac{1}{\sqrt{2}\delta}(S_j-32.4-20\log f_j-R_{\rm{th}})$ and $\beta=10\sqrt{2}/\delta$. Then equation (\ref{App_5}) can be rewritten as
  \begin{eqnarray*}\label{App_7}
\begin{aligned}
P_i^j(t)&=\frac{1}{2}+\frac{1}{\sqrt{\pi}}\int_0^{\alpha_j-\beta \log d_{ij}(t)}e^{-x^2}dx\\
&=\frac{1}{\sqrt{\pi}}\int_{-\infty}^{\alpha_j-\beta \log d_{ij}(t)}e^{-x^2}dx\\
&>\frac{1}{\sqrt{\pi}}\exp \left(-(\beta \log d_{ij}(t)-\alpha_j+1)^2\right).
\end{aligned}
 \end{eqnarray*}
From this expression, there exist two positive constants $c_1=c_1(n)$ and $c_2=c_2(n)$ such that
  \begin{eqnarray}\label{App_8}
\begin{aligned}
&P(\mathcal{G}^l \mbox{ is connected})\\
&>c_1 \exp \left(-c_2 (\beta \log d_{\max}^l-\alpha_{\min}+1)^2\right),
\end{aligned}
 \end{eqnarray}
 where $d_{\max}^l:=\max_{1\leq i,j\leq n,lT\leq t<(l+1)T}d_{ij}(t)$ and $\alpha_{\min}:=\min_{1\leq j\leq n} \alpha_j$.
 If there exist two constants $u\in(0,1/2)$ and $U>0$ such that
 \begin{eqnarray}\label{App_9}
\begin{aligned}
d_{\max}^l\leq \exp \left(\frac{\sqrt{u\log l-\log\log l+U}}{\beta\sqrt{c_2}}+\frac{\alpha_{\min}-1}{\beta}  \right)
\end{aligned}
 \end{eqnarray}
 for any $l\geq 1$, then from (\ref{App_8}) we have
   \begin{eqnarray}\label{App_10}
\begin{aligned}
P(\mathcal{G}^l \mbox{ is connected})>c_1 e^{-U} l^{-u} \log l, ~~\forall l\geq 1.
\end{aligned}
 \end{eqnarray}
We assume that the topologies $\{\mathcal{G}^l\}$ are mutually independent. Consequently, (\ref{App_10})
implies that the topologies satisfy (A1'). Recall that $\mathcal{G}^l$ is undirected and
 $\{\xi_j^l+\zeta_i^{j,l}\}$ is assumed to be a zero-mean and bounded variance $\tilde{\rho}$-mixing sequence.
By Theorem \ref{rand_topo},  system (\ref{App_4}) reaches  unbiased mean square average-consensus if we choose a suitable $a(l)$.
Also, by Remark \ref{rem_A2_a} we can control the final consensus value to be arbitrarily close to the average value of the initial state.

Inequality (\ref{App_9}) claims that to guarantee convergence to consensus the distance between agents cannot grow too fast. In fact, (\ref{App_9}) can be satisfied if  the velocity difference $\|V_i(t)-V_j(t)\|=O(\frac{1}{t})$ for any $1\leq i,j\leq n$, where $V_i(t)$ denotes the velocity of agent $i$ at time $t$.
Of course, the consensus speed depends not only on the growth rate of the distance between agents but also the initial distance.
On the other hand, if there exists a constant $b\in[0,1)$ such that
$\|V_i(t)-V_j(t)\|\geq \Theta(\frac{1}{t^{b}})$, then inequality (\ref{App_9}) may not be satisfied.

\subsection{Simulations}
We now perform simulations to evaluate distributed average consensus of mobile ad-hoc networks. Assume that nine agents are moving in the plane with velocities
$V(t)+v_i(t)$, $1\leq i\leq 9, t\geq 0$, where $V(t)$ is the velocity of the whole group at time $t$, and $v_i(t)$ is the relative velocity of agent $i$.
 Suppose that the initial position of agent $i$ is $(\cos\frac{(i-1)\pi}{8},\sin\frac{(i-1)\pi}{8})$, and the heading of the relative velocity $v_i(t)$ of agent $i$
is the constant $\frac{(i-1)\pi}{8}$.
The initial positions and headings of the relative velocities of all agents are shown in Fig. \ref{shownode}.

\myfigure{\includegraphics[width=.8\columnwidth,height=.4\columnwidth]{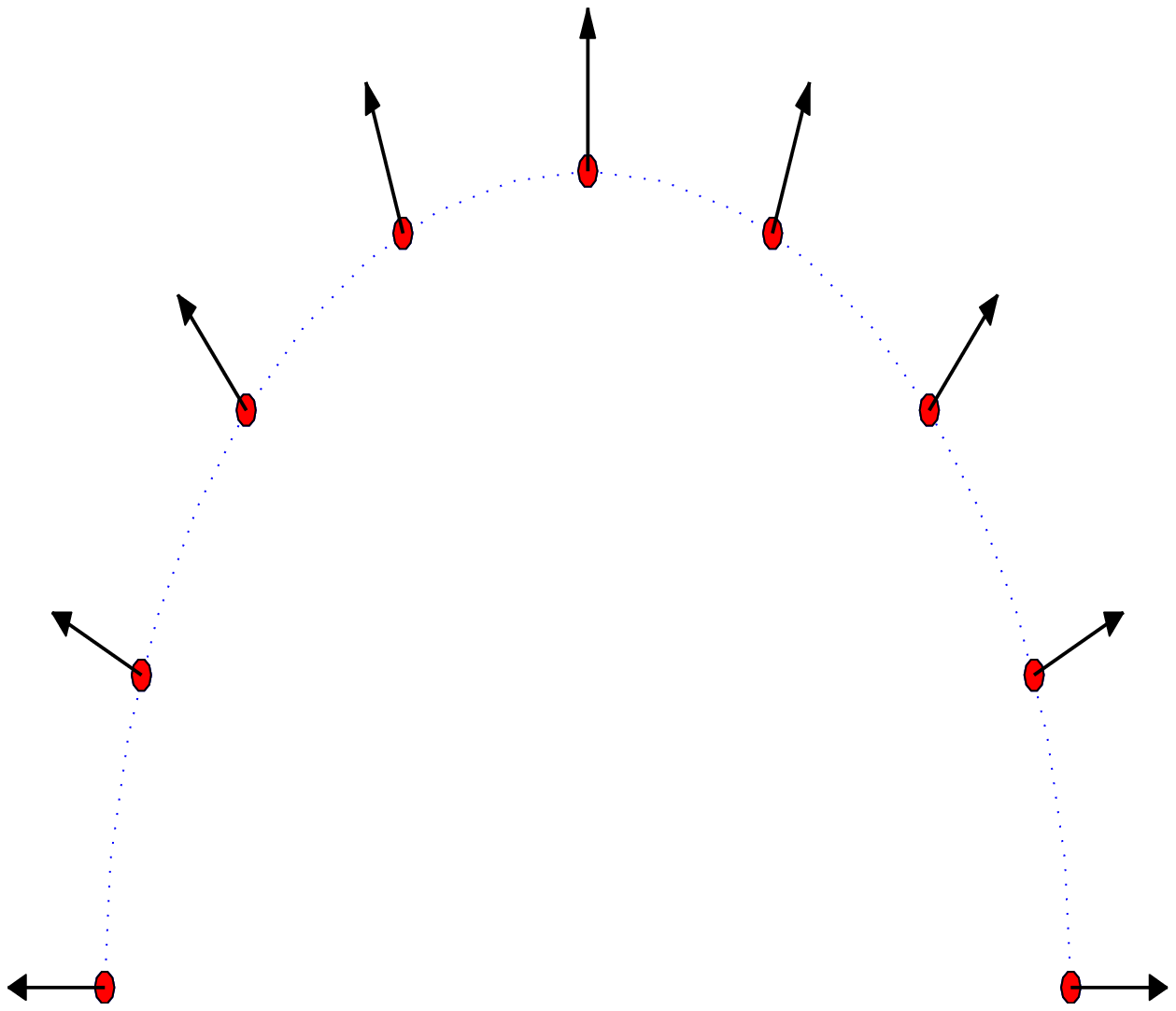}%
\figcaption{\emph{The initial positions and relative movement directions of all agents.}}\label{shownode}}

Each agent $i$ contains a state $x_i(t)\in\mathbb{R}$ which is initially set to be $\frac{i-1}{8}$.

We adopt the consensus protocol in Subsection \ref{subsec_app}, which means that the states of all agents are updated by (\ref{App_4}).
The period length $T$ is selected to be $1$. Let $\delta=1, \beta=10\sqrt{2}$, and $\alpha_j=4$ for $1\leq j\leq 9$,  where $\delta, \beta$, and $\alpha_j$ are the same constants appearing in Subsection \ref{subsec_app}. Assume that the
quantization noise $\{\xi_j^l\}$ is independent and uniformly distributed in $[-\frac{1}{16},\frac{1}{16}]$, and the reception error $\{\zeta_i^{j,l}\}$
obeys a Gaussian distribution whose expectation is zero and standard deviation is $0.05$.

We first simulate the case where the relative velocity magnitude $\|v_i(t)\|$ of agent $i$ equals $\frac{1}{t+200}$ for all $1\leq i\leq 9$ and
$t\geq 0$. In this case (\ref{App_9}) holds for any $\mu>0$ if we choose a suitable $U$.  In consideration of the consensus speed
and its variance, we select $a(t)=\frac{1}{(t+30)^{0.99}}$ according to Theorem \ref{rand_topo} and Remark \ref{rem_A2_a}.
Fig. \ref{FigS1} shows a simulation result in which the states of all agents converge to a consensus value close to $0.5=\frac{1}{9}\sum_{i=1}^9 x_i(0)$.

\myfigure{\includegraphics[width=.8\columnwidth]{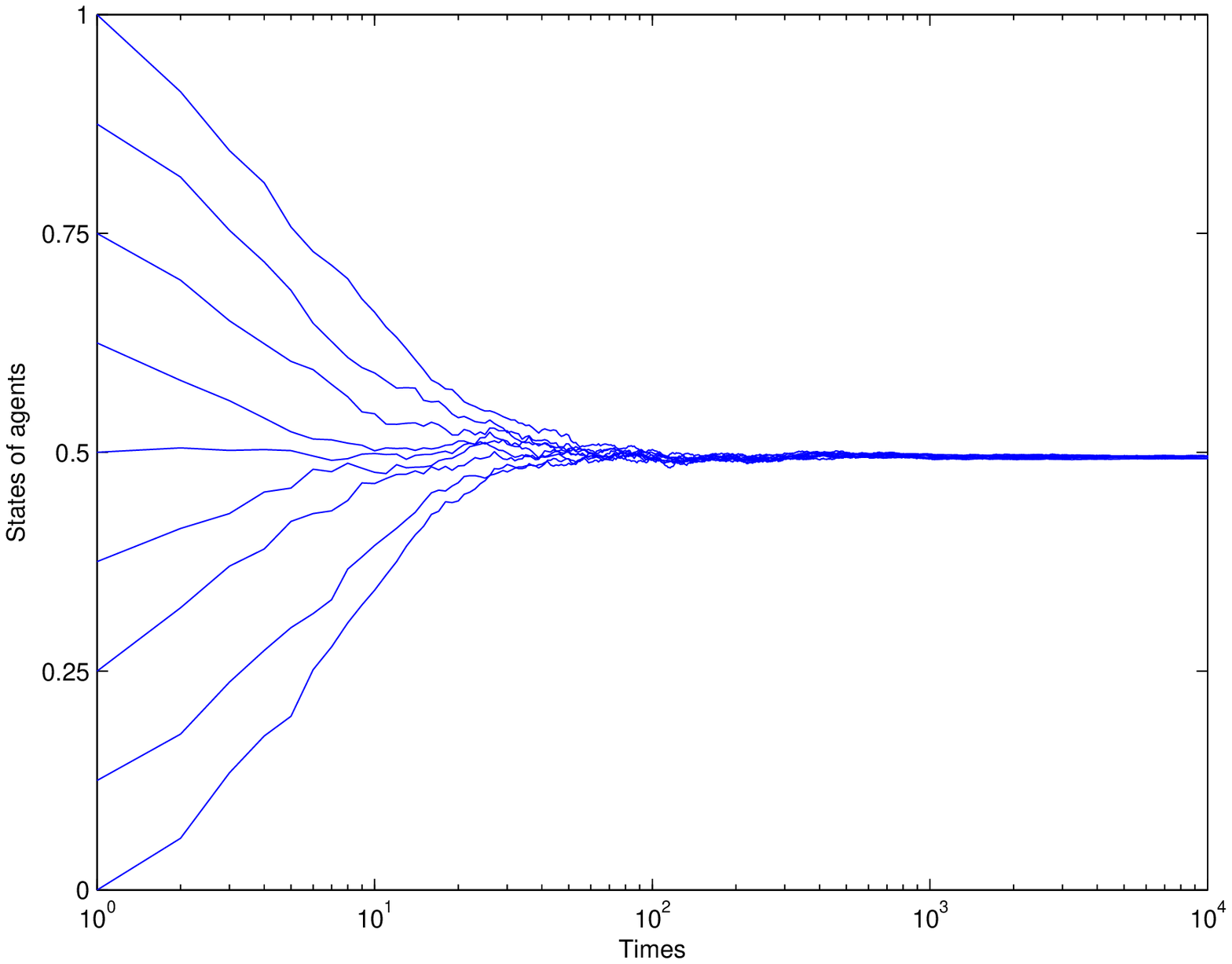}%
\figcaption{\emph{The x-axis denotes the time $t$, while y-axis denotes the state $x_i(t)$ of each agent $i$. For all $t\geq 0$,
the control gain $a(t)$ is set to be $\frac{1}{(t+30)^{0.99}}$, and
the  magnitude of the relative velocity of every agent equals to $\frac{1}{t+200}$, i.e., $\|v_i(t)\|=\frac{1}{t+200}$, $1\leq i\leq 9$.}}\label{FigS1}}


An interesting problem is: does protocol (\ref{App_4}) still achieve consensus if  (\ref{App_9}) is not satisfied?  We increase
 the relative velocity magnitude to $\frac{1}{(t+200)^{0.9}}$ which violates (\ref{App_9}) for any  $\mu>0$.
For comparison,
all other simulations are kept same.
The simulation result is shown in Fig. \ref{FigS2}.

\myfigure{\includegraphics[width=.8\columnwidth]{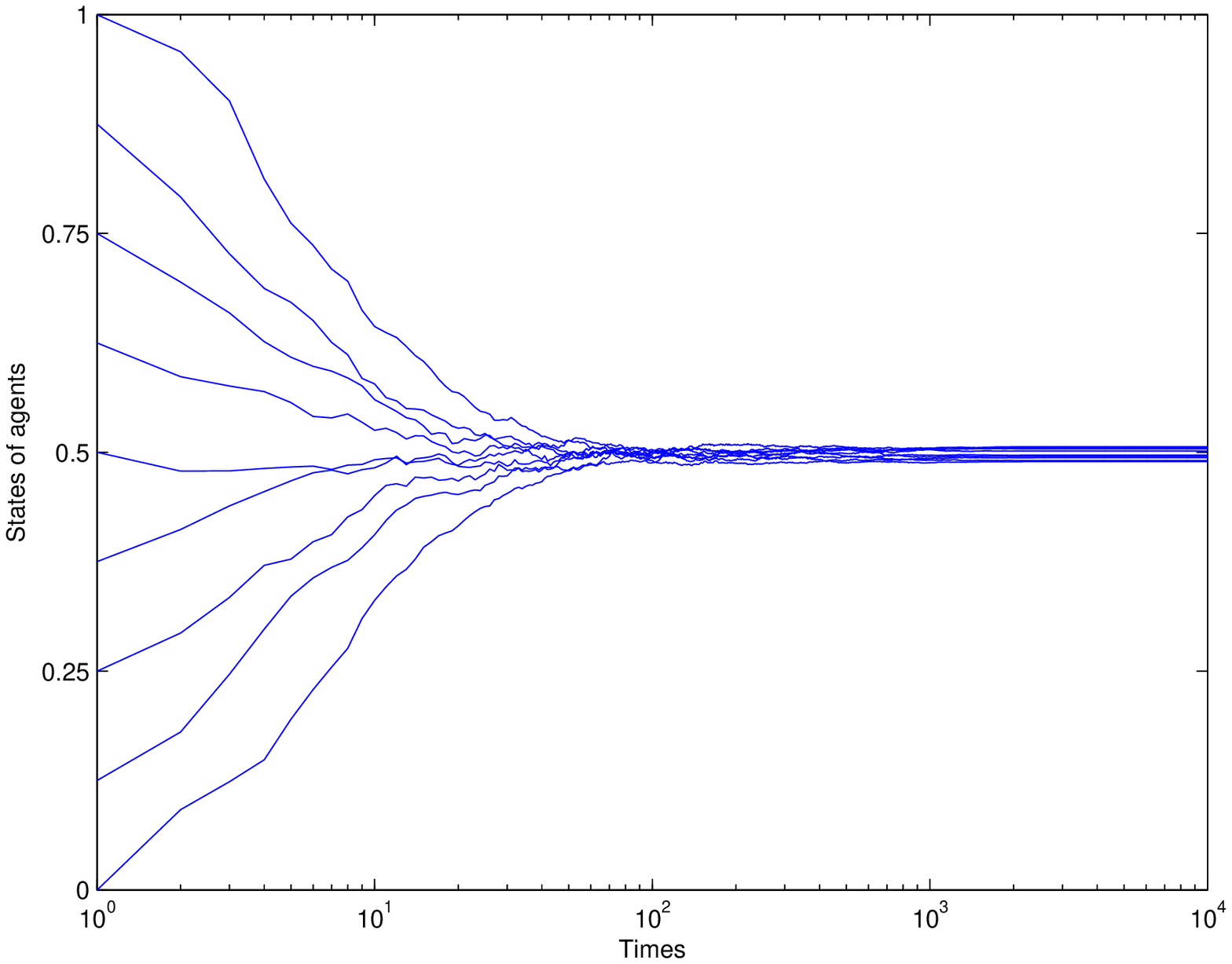}%
\figcaption{\emph{The magnitude of the relative velocity of every agent equals to $\frac{1}{(t+200)^{0.9}}$, while the other configuration is as same as Fig. \ref{FigS1}.}}\label{FigS2}}

From this simulation it can be seen that the final states of the agents have a gap. If
 the relative velocity magnitude grows to $\frac{1}{(t+200)^{0.8}}$, the gap between the agents' final states become more significant,
 see Fig. \ref{FigS3}.

 \myfigure{\includegraphics[width=.8\columnwidth]{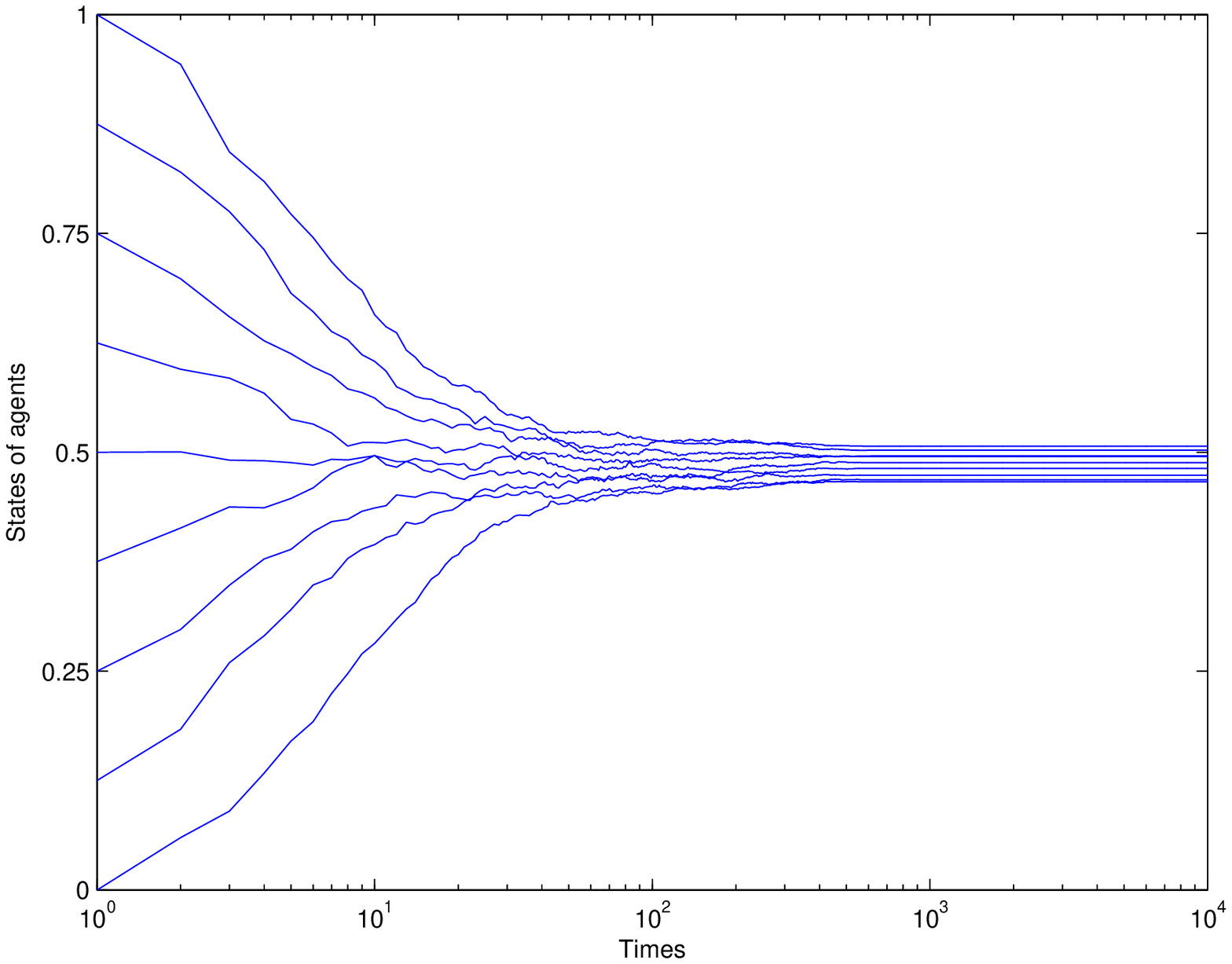}%
\figcaption{\emph{The magnitude of the relative velocity of every agent equals  $\frac{1}{(t+200)^{0.8}}$, while all other conditions are kept same as Fig. \ref{FigS1}.}}\label{FigS3}}

From above simulations it is conjectured that
if the magnitude of the relative velocity of every agent is  $\frac{1}{(t+200)^{b}}$, the critical value of $b$ equals  $1$ for consensus.

\section{Conclusions}\label{Conclusions}
Consensus behavior of multi-agent systems has drawn substantial interests over the past two
decades. However,
some key problems remain unsolved, including the fastest convergence
speeds and critical consensus conditions of connectivity on network topologies. This paper
addresses these problems based on a first-order average-consensus protocol with switching
topologies and additive noises. We first propose an extensible joint-connectivity condition
on topologies.
Using stochastic approximation methods and under our new
condition, we establish a critical consensus condition for network topologies, and provide the fastest convergence rates with respect to the best and worst topologies. Our results give a quantitative description of  the relation between convergence speed
and connectivity of network topologies. Also, we give a consensus analysis for our systems with non-stationary and strongly correlated stochastic topologies,
and apply it to distributed consensus of mobile ad-hoc networks.

\appendices
\section{}\label{App_lemmas}

\begin{proof}[Proof of Lemma \ref{Lemma_3}]
Take $A_t=I-a(t) L(t)$. Using (A2) and the condition of $a(t)\in (0,1/d_{\max})$,
 we have that
$A_t$ is a doubly stochastic matrix.
Also, we can compute
 \begin{eqnarray*}
(A_t' A_t)_{ij}\geq (1-a(t)d_{\max} ) \left[(A_t)_{ij}+(A_t)_{ji}\right]
\end{eqnarray*}
for all $t\geq 1$, and from (A1) and (\ref{a_assum}) we have
\begin{eqnarray*}
&&\min_{\emptyset\subset S\subset \{1,\ldots,r\}} \sum_{i\in S, j\in S^c}  \sum_{t=t_{l}}^{t_{l+1}-1} \left[(A_t)_{ij}+(A_t)_{ji}\right]\\
&& \geq  \min_{t_{l}\leq t <t_{l+1} }a(t)=\delta_l
\end{eqnarray*}
 With these it is deduced directly
from the proof of Theorem 6 in \cite{Touri2011} that
\begin{eqnarray}\label{Lemma_3_1}
&&V\left(z(t_{\widetilde{k}^t}-1)\right)\nonumber\\
&&=V\left(\Phi(t_{\widetilde{k}^t}-1,t_{\widetilde{k}^t-1})\cdots \Phi(t_{k^i+1}-1,t_{k^i})z(t_{k^i}-1)\right)\nonumber\\
&&\leq V\left(z(t_{k^i}-1)\right) \prod_{l=k^i}^{\widetilde{k}^t-1} \left(1-\frac{\delta_l(1-\delta_l)^2 \varepsilon_l }{n(n-1)^2} \right)
\end{eqnarray}
with  $\varepsilon_l=\min_{t_{l}\leq t <t_{l+1} }(1-a(t) d_{\max})$.
By Theorem 5 in \cite{Touri2011}, we get $V(z(t))\leq V (z(t_{\widetilde{k}^t}-1))$ and $V(z(t_{k^i}-1))\leq V(z(i))$. Combining these and
(\ref{Lemma_3_1}) yields our result.
\end{proof}

\begin{proof}[Proof of Lemma \ref{Lemma_rate1}]
First, we use induction to prove that for all integer $j\geq -1$,
\begin{eqnarray}\label{Lemma_r1_2}
t_{k^i+j}\leq \left(c(j+1)+i^{1-\delta} \right)^{1/(1-\delta)}.
\end{eqnarray}
 By the definition of $k^i$ we have $t_{k^i-1}\leq i$, so
(\ref{Lemma_r1_2}) holds for $j=-1$.  Also, if (\ref{Lemma_r1_2}) holds for $j\geq -1$, then
\begin{eqnarray*}\label{Lemma_r1_3}
&&t_{k^i+j+1}\leq t_{k^i+j}+c t_{k^i+j}^{\delta}\\
&&\leq \Big(c(j+1)+i^{1-\delta} \Big)^{\frac{1}{1-\delta}}+c\Big(c(j+1)+i^{1-\delta} \Big)^{\frac{\delta}{1-\delta}}\\
&&=c^{\frac{1}{1-\delta}}\Big(j+1+\frac{i^{1-\delta}}{c} \Big)^{\frac{1}{1-\delta}}\Big[1+\left(j+1+\frac{i^{1-\delta}}{c}\right)^{-1}  \Big]\\
&&\leq c^{\frac{1}{1-\delta}}\big(j+1+\frac{i^{1-\delta}}{c} \big)^{\frac{1}{1-\delta}}\big[1+\big(j+1+\frac{i^{1-\delta}}{c}\big)^{-1}  \big]^{\frac{1}{1-\delta}}\\
&&=\left(c(j+2)+i^{1-\delta} \right)^{\frac{1}{1-\delta}}.
\end{eqnarray*}
Hence, (\ref{Lemma_r1_2}) also holds for $j+1$. By induction we have (\ref{Lemma_r1_2}) holds for all $j\geq -1$.

Thus, using (\ref{Lemma_r1_2}) we get
\begin{eqnarray*}\label{Lemma_r1_4}
t_{k^i+\lfloor \frac{(t+1)^{1-\delta}-i^{1-\delta}}{c}\rfloor-1}\leq \left((t+1)^{1-\delta} \right)^{1/(1-\delta)}=t+1.
\end{eqnarray*}
From the definition of $\widetilde{t}$,
\begin{eqnarray*}\label{Lemma_r1_5}
k^i+\lfloor \frac{(t+1)^{1-\delta}-i^{1-\delta}}{c}\rfloor -1\leq \widetilde{k}^t.
\end{eqnarray*}
This, together with (\ref{Lemma_r1_2}) and the fact that $\log(1-x)< -x/2$ for $x\in (0,1)$, implies
\begin{eqnarray}\label{Lemma_r1_6}
&&\prod_{j=k^i}^{\widetilde{k}^t-1} \big( 1-\frac{c_1}{ t_{j+1}^{1-\delta}+t^*}\big)\\
&&\leq \prod_{j=0}^{\lfloor \frac{(t+1)^{1-\delta}-i^{1-\delta}}{c}-2\rfloor} \big( 1-\frac{c_1}{t_{k^i+j+1}^{1-\delta}+t^*} \big)\nonumber\\
&&\leq \prod_{j=0}^{\lfloor \frac{(t+1)^{1-\delta}-i^{1-\delta}}{c}-2\rfloor} \big( 1-\frac{c_1}{c(j+2)+i^{1-\delta}+t^*}\big)\nonumber\\
&&< \exp \Big(\sum_{j=0}^{\lfloor \frac{(t+1)^{1-\delta}-i^{1-\delta}}{c}-2\rfloor} \frac{-c_1}{2\left[c(j+2)+i^{1-\delta}+t^* \right]}\Big).\nonumber
\end{eqnarray}
Because for any  $a>0$ and integer $b>0$,
\begin{eqnarray*}
\sum_{k=0}^b \frac{1}{k+a}>\int_0^{b+1}\frac{dx}{x+a}=\log (b+1+a)-\log a,
\end{eqnarray*}
(\ref{Lemma_r1_6}) is followed by
\begin{eqnarray}\label{Lemma_r1_7}
&&\prod_{j=k^i}^{\widetilde{k}^t-1} \left( 1-\frac{c_1}{t_{j+1}^{1-\delta}+t^*} \right)<\exp\Big\{\frac{-c_1}{2c}\nonumber\\
&&~\cdot \Big[\log \frac{(t+1)^{1-\delta}+t^*}{c}-\log \Big(2+\frac{i^{1-\delta}+t^*}{c}\Big) \Big]  \Big\}\nonumber\\
&&=\left(\frac{i^{1-\delta}+2c+t^*}{(t+1)^{1-\delta}+t^*}\right)^{\frac{c_1}{2c}}.
\end{eqnarray}

With the similar process from (\ref{Lemma_r1_6}) to (\ref{Lemma_r1_7}) we get (\ref{Lemma_r1_01}).
\end{proof}

\begin{proof}[Proof of Lemma \ref{Lemma_lower}]
let $y:=(I-aL)x$, $\widetilde{y}:=y-y_{\rm{ave}}\mathds{1}$, $\widetilde{x}=x-x_{\rm{ave}}\mathds{1}$ and $\pi=\frac{1}{n}\mathds{1}'$. Then
\begin{eqnarray*}\label{Lemma_lower1}
\begin{aligned}
\widetilde{y}&=y-(\pi y)\mathds{1}=(I-aL)x-\left[\pi(I-aL)x\right]\mathds{1}\\
&=x-(\pi x)\mathds{1}-
a\left[Lx-(\pi Lx)\mathds{1}  \right]\\
&=\widetilde{x}-a\left[L\widetilde{x}-(\pi L\widetilde{x})\mathds{1}  \right],
\end{aligned}
\end{eqnarray*}
Combining this with $\widetilde{x}'\mathds{1}=0$ we get
\begin{eqnarray*}\label{Lemma_lower2}
\begin{aligned}
&V(y)=\|\widetilde{y}\|^2=\widetilde{y}' \widetilde{y}\\
 &\geq\|\widetilde{x}\|^2-a\widetilde{x}' \left[L\widetilde{x}-(\pi L\widetilde{x})\mathds{1}  \right]-a\left[L\widetilde{x}-(\pi L\widetilde{x})\mathds{1}  \right]' \widetilde{x}\\
&=\|\widetilde{x}\|^2-a \widetilde{x}' (L+L') \widetilde{x}\\
&\geq \|\widetilde{x}\|^2- a \lambda_{\max}(L+L') \|\widetilde{x}\|^2\\
&=\left(1-a\lambda_{\max}(L+L')\right) V(x).
\end{aligned}
\end{eqnarray*}
\end{proof}

\section{proof of Theorem \ref{nece1}}\label{App_nece1}
 We will show  system
(\ref{model0})-(\ref{model1}) cannot reach consensus in mean square by contradiction. Because to reach consensus $a(t)$ must converge to $0$,
there exists an integer $k_1$ such that $a(t)<\frac{1}{2n}$ for all $t\geq t_{k_1}$. We take $t^*=t_{k_1}$.
Similar to (\ref{the_1_2}) we get
\begin{eqnarray}\label{nece1_1}
x(t+1)=\Phi(t,t^*)x(t^*)+\sum_{i=t^*}^{t}a(i)\Phi(t,i+1)\widehat{w}(i).
\end{eqnarray}

We choose $t_k=t_{k-1}+c \lfloor t_{k-1}^{\delta}\rfloor$ and select
\begin{eqnarray*}\label{nece1_4}
\mathcal{G}(t):=
\left\{%
\begin{array}{ll}
\mathcal{G}_1,~~~~\mbox{if } t\in \cup_{k=1}^{\infty} \{t_k^*\},\\
\mathcal{G}_2,~~~~\mbox{otherwise},
\end{array}%
\right.
\end{eqnarray*}
where $t_k^*:=\arg\min_{t_k\leq t<t_{k+1}}a(t)$.
Here if there are more than one time reach $\min_{t_k\leq t<t_{k+1}}a(t)$ then we randomly pick one as $t_k^*$.
Let $\pi=\frac{1}{n}\mathds{1}'$, then we can compute $\pi L(t)=\textbf{0}$. So our choice satisfies both (A1) and (A2).

Set $S:=\cup_{k=1}^{\infty} \{t_k^*\}$ and let
\begin{eqnarray}\label{nece1_5_b}
b_i^t:=\prod_{j\in S\cap [i,t]}(1-na(j)), c_i^t:=\prod_{j\in S^c \cap [i,t]}(1-2a(j)).
\end{eqnarray}
By Proposition \ref{prop_matrix},
$\Phi(t,i)=P\mbox{diag}\left(1,b_i^t c_i^t, b_i^t,\ldots,b_i^t \right) P'.$
Set
\begin{eqnarray}\label{nece1_6_a1}
\widetilde{\Phi}(t,i):=P\mbox{diag}\left(0,b_i^t c_i^t, b_i^t,\ldots,b_i^t \right) P'
\end{eqnarray}
and $x_{\rm{ave}}(t):=\frac{1}{n}\sum_{i=1}^n x_i(t)$ be the average value of $x(t)$,
by (\ref{nece1_1}) we get
\begin{eqnarray}\label{nece1_6_a2}
&&x(t+1)-x_{\rm{ave}}(t+1)\mathds{1}=x(t+1)-(\pi x(t+1))\mathds{1}\nonumber\\
&&=\Phi(t,t^*)x(t^*)-(\pi x(t^*))\mathds{1}+a(t)\left[\widehat{w}(t)-(\pi \widehat{w}(t))\mathds{1} \right]\nonumber\\
&&~~+\sum_{i=t^*}^{t-1}a(i)\left[\Phi(t,i+1)\widehat{w}(i)-\left(\pi \widehat{w}(i)\right)\mathds{1} \right]\nonumber\\
&&=\widetilde{\Phi}(t,t^*)x(t^*)+\sum_{i=t^*}^{t-1}a(i)\widetilde{\Phi}(t,i+1)\widehat{w}(i)\nonumber\\
&&~~+a(t)\left[\widehat{w}(t)-(\pi \widehat{w}(t))\mathds{1} \right].
\end{eqnarray}
Because the noises $\{w_{ji}(t)\}$ satisfy (A4),
\begin{eqnarray}\label{nece1_6_a3}
&&E\left[V(x(t+1))\right]=\sum_{i=t^*}^{t-1}a^2(i)E\|\widetilde{\Phi}(t,i+1)\widehat{w}(i)\|^2\nonumber\\
&&~~+E\|\widetilde{\Phi}(t,t^*)x(t^*)\|^2+a^2(t)E\|\widehat{w}(t)-(\pi \widehat{w}(t))\mathds{1}\|^2\nonumber\\
&&\geq \sum_{i=t^*}^{t-1}a^2(i)E\|\widetilde{\Phi}(t,i+1)\widehat{w}(i)\|^2.
\end{eqnarray}
Take $y=\widetilde{\Phi}(t,i+1)\widehat{w}(i)$. We can compute
\begin{eqnarray*}\label{nece1_7}
&&y_1=b_{i+1}^t\\
&&\cdot\left[\frac{\widehat{w}_1(i)-\widehat{w}_2(i) }{2} c_{i+1}^t+\frac{\widehat{w}_1(i)+\widehat{w}_2(i) }{2}-\pi \widehat{w}(i)\right],
\end{eqnarray*}
so by (A4) there exists a constant $c'>0$ such that
\begin{eqnarray}\label{nece1_8}
E\|y\|^2 \geq E y_1^2 =c' (b_{i+1}^t)^2.
\end{eqnarray}
Substituting this into (\ref{nece1_6_a3}),  we
get
\begin{eqnarray}\label{nece1_9}
\begin{aligned}
E\left[V(x(t+1))\right]\geq c'\sum_{i=t^*}^{t-1}a^2(i)(b_{i+1}^t)^2
\end{aligned}
\end{eqnarray}
To reach consensus  the last line of (\ref{nece1_9}) should converge to $0$. We
show that this
is impossible.

First, because the initial state $x(1)$ is not consistent, to reach consensus we must choose some positive $a(t)$. Considering the affection of  noises, we must select positive $a(t)$ infinite times to guarantee consensus in mean square. Thus, we can pick $t'\geq t^*$ such that
$a(t')>0$. If the last line of (\ref{nece1_9}) converge to $0$, we have
\begin{eqnarray}\label{nece1_10}
\prod_{k=k_1}^{\infty}(1-na(t_k^*))=\lim_{t\rightarrow\infty} b_{t^*}^t\leq\lim_{t\rightarrow\infty} b_{t'+1}^t =0.
\end{eqnarray}
Also, we recall that $t_{\widetilde{k}^t}\leq t+1$, so
\begin{eqnarray}\label{nece1_11}
\begin{aligned}
&\sum_{i=t^*}^{t}a^2(i)(b_{i+1}^{t+1})^2\geq \sum_{k=k_1}^{\widetilde{k}^t-1}\sum_{i=t_k}^{t_{k+1}-1} a^2(i)(b_{i+1}^{t+1})^2\\
&\geq \sum_{k=k_1}^{\widetilde{k}^t-1}\sum_{i=t_k}^{t_{k+1}-1} a^2(i)\prod_{j=k}^{\widetilde{k}^t}\left[1-na(t_j^*)\right]^2\\
&\geq \sum_{k=k_1}^{\widetilde{k}^t-1}(t_{k+1}-t_k) a^2(t_k^*)\prod_{j=k}^{\widetilde{k}^t}\left[1-na(t_j^*)\right]^2\\
&>\frac{1}{16}\sum_{k=k_1}^{\widetilde{k}^t-1} \lfloor c t_k^{\delta}\rfloor a^2(t_k^*)\prod_{j=k+1}^{\widetilde{k}^t-1}\left[1-2na(t_j^*)\right].
\end{aligned}
\end{eqnarray}
Let $I_{\{\cdot\}}$ be the indicator function, then
\begin{eqnarray}\label{nece1_12}
\begin{aligned}
1-2na(t_j^*)&=\left[1-2na(t_j^*) I_{\{a(t_j^*)>t_j^{-\delta}\}}\right]\\
&~\cdot\left[1-2na(t_j^*) I_{\{a(t_j^*)\leq t_j^{-\delta}\}}\right].
\end{aligned}
\end{eqnarray}
Because $c\geq 1$ and $\delta>1/2$, by the choice of $t_k$, we have
\begin{eqnarray*}
\begin{aligned}
t_k&=t_{k-1}+c \lfloor t_{k-1}^{\delta} \rfloor\\
&\geq t_{k-1}+\frac{c}{2} t_{k-1}^{\delta}> t_{k-1}+\frac{\sqrt{t_{k-1}}}{2},
\end{aligned}
 \end{eqnarray*}
 then by induction we get $t_k>\frac{1}{20}k^2$. So,
\begin{eqnarray}\label{nece1_13}
\begin{aligned}
&\sum_{j=k_1}^{\infty} a(t_j^*) I_{\{a(t_j^*)\leq t_j^{-\delta}\}}\leq \sum_{j=k_1}^{\infty} t_j^{-\delta}\\
&~<\sum_{j=k_1}^{\infty} 20^{\delta} j^{-2\delta}<\infty,
\end{aligned}
\end{eqnarray}
which indicates
\begin{eqnarray*}\label{nece1_14}
c_1:=\prod_{j=k_1}^{\infty}\left[1-2na(t_j^*) I_{\{a(t_j^*)\leq t_j^{-\delta}\}}\right]>0.
\end{eqnarray*}
Substituting this and (\ref{nece1_12}) into (\ref{nece1_11}) and taking $$d_k^t=\prod_{j=k+1}^{\widetilde{k}^t-1}[1-2na(t_j^*) I_{\{a(t_j^*)>t_j^{-\delta}\}}]$$ we have
\begin{eqnarray}\label{nece1_15}
\begin{aligned}
&\sum_{i=t^*}^{t}a^2(i)(b_{i+1}^{t+1})^2\geq \frac{1}{16}\sum_{k=k_1}^{\widetilde{k}^t-1} \lfloor c t_k^{\delta}\rfloor a^2(t_k^*)d_k^t\\
&~~\cdot \prod_{j=k+1}^{\widetilde{k}^t-1}\left[1-2na(t_j^*) I_{\{a(t_j^*)\leq t_j^{-\delta}\}}\right]\\
&\geq \frac{c_1}{16}\sum_{k=k_1}^{\widetilde{k}^t-1} \lfloor c t_k^{\delta}\rfloor a^2(t_k^*)d_k^t\\
&\geq \frac{c_1}{16}\sum_{k=k_1}^{\widetilde{k}^t-1} \lfloor c t_k^{\delta}\rfloor a^2(t_k^*)I_{\{a(t_k^*)>t_k^{-\delta}\}}d_k^t\\
&> \frac{c_1}{16}\sum_{k=k_1}^{\widetilde{k}^t-1} \lfloor c t_k^{\delta}\rfloor t_k^{-\delta} a(t_k^*)I_{\{a(t_k^*)>t_k^{-\delta}\}}d_k^t\\
&\geq \frac{c_1c}{32}\sum_{k=k_1}^{\widetilde{k}^t-1} a(t_k^*)I_{\{a(t_k^*)>t_k^{-\delta}\}}d_k^t\\
&=\frac{c_1c}{64n}\left(1-\prod_{j=k_1}^{\widetilde{k}^t-1} \left[1-2na(t_j^*) I_{\{a(t_j^*)>t_j^{-\delta}\}}\right] \right).
\end{aligned}
\end{eqnarray}
Also, by (\ref{nece1_10}) and (\ref{nece1_13}) we have
$\sum_{j=k_1}^{\infty}a(t_j^*) I_{\{a(t_j^*)>t_j^{-\delta}\}}=\infty,$
so by (\ref{nece1_15}) we get
\begin{eqnarray*}\label{nece1_16}
\begin{aligned}
\liminf_{t\rightarrow\infty}\sum_{i=t^*}^{t}a^2(i)(b_{i+1}^{t+1})^2\geq \frac{c_1c}{64n}.
\end{aligned}
\end{eqnarray*}
Combining this with (\ref{nece1_9}) we see the system cannot reach consensus in mean square.

\section{proof of Theorem \ref{crit2}}\label{App_crit2}

Without loss of generality, we assume
\begin{eqnarray}\label{crit2_2}
\sum_{i=3}^n[x_i(1)-x_{\rm{ave}}(1)]^2\geq \frac{n-2}{n}V(x(1))>0.
\end{eqnarray}
Choose $t_k$ and $\mathcal{G}(t)$ as same as the proof of Theorem \ref{nece1}, and take
$\pi=\frac{1}{n}\mathds{1}'$. Also, similar to (\ref{Lemma_r1_2}) we can prove that there exists a constant $c_1:=c_1(c,\delta)>0$ such that
\begin{eqnarray}\label{crit2_3}
c_1 k^{1/(1-\delta)}\leq t_k \leq (ck)^{1/(1-\delta)},~~~~\forall k\geq 1.
\end{eqnarray}

For the time $t+2$ with $t\geq 0$, we consider all the choices of $\{a(i)\}_{i=1}^{t+1}$  to get a lower bound of $E\|x(t+2)-x_{\rm{ave}}(t+2)\mathds{1}\|^2$. Let $t^*\in [1,t+2]$ be the minimum time such that if $k\geq t^*$ then $a(k)<\frac{1}{3n}$. We see if $t^*\geq 2$ then $a(t^*-1)\geq \frac{1}{3n}$.
By (\ref{nece1_6_a3}) we have
\begin{eqnarray}\label{crit2_4}
\begin{aligned}
&E\left[V(x(t+2))\right]=E\|\widetilde{\Phi}(t+1,t^*)x(t^*)\|^2\\
&~~+\sum_{i=t^*}^{t}a^2(i)E\|\widetilde{\Phi}(t+1,i+1)\widehat{w}(i)\|^2\\
&~~ +a^2(t+1)E\|\widehat{w}(t+1)-(\pi \widehat{w}(t+1))\mathds{1}\|^2.
\end{aligned}
\end{eqnarray}
If $t^*=t+2$ which means $t^*-1=t+1$,  by (\ref{crit2_4}) and (\ref{cap_1})
$E\left[V(x(t+2))\right]\geq a^2(t^*-1) V(\widehat{w}(t^*-1)) \geq \frac{c_1}{9n^2}$,
which is followed by our result directly.

It remains to consider the case of $t^*< t+2$.
Recall that $b_i^t:=\prod_{j\in S\cap [i,t]}(1-na(j))$ defined in (\ref{nece1_5_b}). For any $y=\widetilde{\Phi}(t,i)x$ we can compute
$y_j=\left(x_j-\pi x\right)b_i^t$ for any $3\leq j\leq n$,
so if $t^*=1$ then
\begin{eqnarray}\label{crit2_6}
&&E\|\widetilde{\Phi}(t+1,t^*)x(t^*)\|^2 \geq \left(b_{t^*}^{t+1}\right)^2\sum_{j=3}^n[x_j(1)-x_{\rm{ave}}(1)]^2\nonumber\\
&&~\geq  (n-2)\left(b_{t^*}^{t+1}\right)^2V(x(1))/n,
\end{eqnarray}
where the last inequality uses (\ref{crit2_2}). Otherwise, by the choice of $t^*$ we have $a(t^*-1)\geq \frac{1}{3n}$, so
\begin{eqnarray}\label{crit2_7}
\begin{aligned}
&E\|\widetilde{\Phi}(t+1,t^*)x(t^*)\|^2\geq a^2(t^*-1)\\
&\cdot E\|\widetilde{\Phi}(t+1,t^*)\widehat{w}(t^*-1)\|^2\geq \frac{c'}{9n^2}(b_{t^*}^{t+1})^2,
\end{aligned}
\end{eqnarray}
where the last inequality uses (\ref{nece1_8}). Set $k_1:=k^{t^*-1}$, then we have $t_{k_1}\geq t^*$ but $t_{k_1-1}<t^*$.
With the similar process from (\ref{nece1_8}) to (\ref{nece1_11}), there exists a constant $c_2>0$ such that
\begin{eqnarray}\label{crit2_8}
&&\sum_{i=t^*}^{t}a^2(i)E\|\widetilde{\Phi}(t+1,i+1)\widehat{w}(i)\|^2 \geq c' \sum_{i=t^*}^{t}a^2(i)(b_{i+1}^{t+1})^2\nonumber\\
&&\geq c_2 \sum_{k=k_1}^{\widetilde{k}^t-1} \lfloor c t_k^{\delta}\rfloor a^2(t_k^*)\prod_{j=k+1}^{\widetilde{k}^t-1}\left[1-2na(t_j^*)\right].
\end{eqnarray}
Similar to (\ref{nece1_12}) we have
\begin{eqnarray*}\label{crit2_9}
\begin{aligned}
1-2na(t_j^*)&=\left[1-2na(t_j^*) I_{\{a(t_j^*)>t^{2\delta-1}/\lfloor c t_j^{\delta}\rfloor \}}\right] \\
&~~\cdot\left[1-2na(t_j^*) I_{\{a(t_j^*)\leq t^{2\delta-1}/\lfloor c t_j^{\delta}\rfloor \}}\right].
\end{aligned}
\end{eqnarray*}
According to (\ref{crit2_3}) and the fact
$t_{\widetilde{k}^t}\leq t+1$,  we
get
\begin{eqnarray}\label{crit2_10}
\begin{aligned}
&\sum_{j=1}^{\widetilde{k}^t-1}a(t_j^*) I_{\{a(t_j^*)\leq t^{2\delta-1}/\lfloor c t_j^{\delta}\rfloor \}}\\
& \leq \sum_{j=1}^{\widetilde{k}^t-1}\frac{t^{2\delta-1}}{\lfloor c t_j^{\delta}\rfloor }=t^{2\delta-1} \sum_{j=1}^{\widetilde{k}^t-1}O\left(j^{\frac{-\delta}{1-\delta}} \right)\\
&=t^{2\delta-1} O\left((\widetilde{k}^t)^{\frac{1-2\delta}{1-\delta}} \right)=O\left(t_{\widetilde{k}^t}^{2\delta-1} (\widetilde{k}^t)^{\frac{1-2\delta}{1-\delta}} \right)\\
& =O\left((\widetilde{k}^t)^{\frac{2\delta-1}{1-\delta}} (\widetilde{k}^t)^{\frac{1-2\delta}{1-\delta}} \right)=O(1),
\end{aligned}
\end{eqnarray}
so with the similar process from (\ref{nece1_13}) to (\ref{nece1_15}),  there exists a constant $c_3>0$ such that
\begin{eqnarray}\label{crit2_12}
&&\sum_{k=k_1}^{\widetilde{k}^t-1} \lfloor c t_k^{\delta}\rfloor a^2(t_k^*)\prod_{j=k+1}^{\widetilde{k}^t-1}\left[1-2na(t_j^*)\right]\\
&&\geq \frac{c_3}{t^{1-2\delta}} \sum_{k=k_1}^{\widetilde{k}^t-1} a(t_k^*)I_{\{a(t_j^*)>t^{2\delta-1}/\lfloor c t_j^{\delta}\rfloor \}}\nonumber\\
&&~~\cdot\prod_{j=k+1}^{\widetilde{k}^t-1}\left[1-2na(t_j^*) I_{\{a(t_j^*)>t^{2\delta-1}/\lfloor c t_j^{\delta}\rfloor \}}\right]\nonumber\\
&&=\frac{c_3}{2nt^{1-2\delta}}\Big(1- \prod_{j=k_1}^{\widetilde{k}^t-1}\left[1-2na(t_j^*) I_{\{a(t_j^*)>t^{2\delta-1}/\lfloor c t_j^{\delta}\rfloor \}}\right]\Big).\nonumber
\end{eqnarray}
If $$\prod_{j=k_1}^{\widetilde{k}^t-1}[1-2na(t_j^*) I_{\{a(t_j^*)>t^{2\delta-1}/\lfloor c t_j^{\delta}\rfloor \}}]\leq \frac{1}{2},$$ then together (\ref{crit2_12}), (\ref{crit2_8}) and (\ref{crit2_4}) our result is obtained. Otherwise, by the definition of $b_i^t$ and (\ref{crit2_10}) there exists a constant $c_4>0$ such that
\begin{eqnarray*}\label{crit2_13}
\begin{aligned}
b_{t^*}^{t+1} \geq c_4\prod_{j=k_1}^{\widetilde{k}^t-1}\left[1-na(t_j^*) I_{\{a(t_j^*)>t^{2\delta-1}/\lfloor c t_j^{\delta}\rfloor \}}\right]>\frac{c_4}{2}.
\end{aligned}
\end{eqnarray*}
Substituting
this into (\ref{crit2_6}) and (\ref{crit2_7}) we get $E\|\widetilde{\Phi}(t+1,t^*)x(t^*)\|^2$ is bigger than a positive constant, then by
(\ref{crit2_4}) the desired result follows.

\section{proof of Theorem \ref{rand_topo}}\label{App_rand_topo}

Set $t_k=1$ and $t_{k+1}=t_k+\lfloor c t_k^{\mu}\rfloor$, where $c$ is a large constant.
 Let $E_{t_1,t_2}$ be the event of $\bigcup_{k=t_1}^{t_2-1}\mathcal{G}(k)$ is strongly connected.  Then for any
 given topologies $\mathcal{G}(l)$, $1\leq l\leq t_k-1$,
\begin{eqnarray}\label{rand_topo_3}
&&P\left(E_{t_k,t_{k+1}}|\{\mathcal{G}(l)\}_{l=1}^{t_k-1}\right)\nonumber\\
&&\geq P\Big(\bigcup_{i=1}^{\lfloor\frac{t_{k+1}-t_k}{K}\rfloor}E_{t_k+(i-1)K,t_k+iK-1}|\{\mathcal{G}(l)\}_{l=1}^{t_k-1}\Big)\nonumber\\
&&=1- P\Big(\bigcap_{i=1}^{\lfloor\frac{t_{k+1}-t_k}{K}\rfloor}E_{t_k+(i-1)K,t_k+iK-1}^c|\{\mathcal{G}(l)\}_{l=1}^{t_k-1}\Big)\nonumber\\
&&=1- P\left(E_{t_k,t_k+K-1}^c|\{\mathcal{G}(l)\}_{l=1}^{t_k-1}\right)\nonumber\\
&&~~~\cdot\prod_{i=2}^{\lfloor\frac{t_{k+1}-t_k}{K}\rfloor}P\Big(E_{t_k+(i-1)K,t_k+iK-1}^c|\{\mathcal{G}(l)\}_{l=1}^{t_k-1},\nonumber\\
&&~~~~~~~~~~~~~~~~~~~~\bigcap_{j=1}^{i-1}\{ E_{t_k+(j-1)K,t_k+jK-1}^c\}\Big)\nonumber\\
&&>1-\left(1-p t_{k+1}^{-\mu}\log t_k\right)^{\lfloor\frac{t_{k+1}-t_k}{K}\rfloor},
\end{eqnarray}
where the last line uses the assumption (A1').

Let $k^*$ be the minimal time such that $\bigcap_{k=k^*}^{\infty}E_{t_k,t_{k+1}}$ happens, which implies the event
$E_{t_{k^*-1},t_{k^*}}$ does not happen. By (\ref{rand_topo_3}),
\begin{eqnarray}\label{rand_topo_3_1}
\begin{aligned}
&P\left(k^*=k\right)<P\left(E_{t_{k-1},t_{k}}^c \right)\\
&<\left(1-p t_{k}^{-\mu}\log t_{k-1}\right)^{\lfloor\frac{t_{k}-t_{k-1}}{K}\rfloor}<t_k^{-\frac{cp}{2K}}
\end{aligned}
\end{eqnarray}
for large $k$, where the last inequality uses the fact of $t_{k+1}-t_k=\lfloor c t_k^{\mu}\rfloor$ and $\mu<1/2$.
By the total probability theorem,
\begin{eqnarray}\label{rand_topo_4}
\begin{aligned}
&E[V(x(t))]=\sum_{k=1}^{\infty} P(k^*=k)E[V(x(t))|k^*=k].
\end{aligned}
\end{eqnarray}
We need to consider the value of $E[V(x(t))|k^*=k]$. Similar to (\ref{the_1_2_2}) and (\ref{the_1_7}) we get
\begin{eqnarray}\label{rand_topo_4a}
&&E[V(x(t+1))|k^*=k]= E[V\left(\Phi(t,1)x(1)\right)|k^*=k]\nonumber\\
&&+ O\bigg( \sum_{i=1}^{t}a^2(i)E\left[V\left(\Phi(t,i+1)\widehat{w}(i)\right)|k^*=k\right]\bigg),
\end{eqnarray}
and similar to (\ref{crit3_4}) for any $x\in\mathds{R}^n$ we obtain
\begin{eqnarray}\label{rand_topo_5}
&&E\left[V\left(\Phi(t,i+1)x\right)|k^*=k\right]\\
&&\leq E\left[V\left(x\right)\right] \prod_{j=\max\{k^{i},k\}}^{\widetilde{k}^t-1}  \left(1-\frac{4c}{t_{j+1}^{1-\mu}+t^*}\right)\nonumber\\
&&<E\left[V\left(x\right)\right]\left(\frac{2c+(\max\{i,t_{k}-1\})^{1-\mu}+t^*}{(t+1)^{1-\mu}+t^*}\right)^2,\nonumber
\end{eqnarray}
where the last line uses the fact if $k^x=y$ then $x\leq t_{y}-1$.
Since $E[V(\widehat{w}(i))|k=k^*]$ is bounded under (A3),
taking (\ref{rand_topo_5}) into (\ref{rand_topo_4a}) we can obtain
\begin{eqnarray*}\label{rand_topo_6}
\begin{aligned}
&E[V(x(t))|k^*=k]\\
&=O\bigg(\sum_{i=1}^{t_k-1} \frac{a^2(i) t_k^{2(1-\mu)}}{t^{2(1-\mu)}}+\sum_{i=t_k}^{t}\frac{a^2(i) i^{2(1-\mu)}}{t^{2(1-\mu)}}\bigg)\\
&=O\bigg( \frac{t_k^{2(1-\mu)}}{t^{2(1-\mu)}}+\frac{1}{t^{1-2\mu}}\bigg).
\end{aligned}
\end{eqnarray*}
Substitute this and  (\ref{rand_topo_3_1}) into (\ref{rand_topo_4}) we have $E[V(x(t))]=O(\frac{1}{t^{1-2\mu}})$ when $c$ is large enough.

Finally, let $x^*$ be the same value defined by (\ref{the_1_19_b1}). As same as (\ref{the_1_19}) and (\ref{the_1_20}) we can obtain
$E[x^*]=\pi x(1)$ and Var$(x^*)<\infty$.


\section*{Acknowledgment}
Ge Chen and Chen Chen would like to thank the guidance of Prof. Lei Guo from  Academy of Mathematics and Systems Science, Chinese Academy of Sciences.

\end{multicols}

\section*{Authors}

\begin{description}
    \item[Ge Chen] received the B.Sc. degree in mathematics from the University of Science and Technology of China in 2004, and the Ph.D. degree in mathematics from the University of Chinese Academy of Sciences, China, in 2009.

He jointed the  National Center for Mathematics and Interdisciplinary Sciences, Academy of Mathematics and Systems Science, Chinese Academy of Sciences in 2011, and is currently  an Associate Professor. His current research interest is the collective behavior of multi-agent systems.

Dr. Chen received the First Prize of the Application Award from the Operations Research Society of China (2010). One of his papers was selected as a SIGEST paper by the \emph{SIAM Review} (2014). He was also a finalist for the OR in Development prize from the International Federation of Operations Research Societies (2011), and for the best theoretical paper award at the 10th World Congress on Intelligent Control and Automation (2012).

    \item[Le Yi Wang]  received the Ph.D. degree in electrical engineering from McGill University, Montreal, Canada, in 1990.

   Since 1990, he has been with Wayne State University, Detroit, Michigan, where he is currently a Professor in the Department of Electrical and  Computer Engineering. His research interests are in the areas of complexity and information, system identification, robust control, H-infinity optimization, time-varying systems, adaptive systems, hybrid and nonlinear systems, information processing and learning, as well as medical, automotive, communications, power systems, and computer applications of control methodologies.

   He was a keynote speaker in several international conferences. He serves on the IFAC Technical Committee on Modeling, Identification and Signal Processing. He was an Associate Editor of the IEEE Transactions on Automatic Control and several other journals.  He was a Visiting Faculty at University of Michigan in 1996 and Visiting Faculty Fellow at University of Western Sydney in 2009 and 2013. He is a member of the Core International Expert Group at Academy of Mathematics and Systems Science, Chinese Academy of Sciences, and an International Expert Adviser at Beijing Jiao Tong University. He is a Fellow of IEEE.

    \item[Chen Chen] was born in Shannxi, China, in 1987. She received the B.S. degree in Mathematics from Beihang University in 2009, and the Ph.D. degree in Control Theory from the Academy of Mathematics and Systems Science, Chinese Academy of Sciences, in 2014. She is currently a researcher at Huawei Technologies Co. Ltd. Her research interests include complex systems, distributed filters, statistic machine learning. and deep learning.

   \item[George Yin]
   received the B.S. degree in Mathematics from
the University of Delaware in 1983, M.S. degree in Electrical
Engineering, and Ph.D. in Applied Mathematics from
Brown University in 1987.

He joined Wayne State University
in 1987, and became a professor in 1996. His research
interests include stochastic systems and applications.

He served as a member of the program committee for many IEEE
Conference on Control and Decision;
he also severed on the IFAC Technical Committee on Modeling,
Identification and Signal Processing, and many conference
program committees; he was Co-Chair of SIAM Conference
on Control \& Its Application, 2011, and Co-Chair of
two AMS-IMS-SIAM Summer Research Conferences;
he chaired a number SIAM prize selection committees. He was Chair of SIAM Activity
Group on Control and Systems Theory, and served on the Board of Directors of
American Automatic Control Council. He is an associate editor of SIAM Journal on
Control and Optimization, and on the editorial board of a number of other journals.
He was an Associate Editor of Automatica (2005-2011) and IEEE Transactions on
Automatic Control (1994-1998). He is a Fellow of IEEE, Fellow of IFAC, and  Fellow of SIAM.
\end{description}

\end{document}